\documentclass[a4paper]{article}

\usepackage[utf8]{inputenc}
\usepackage[T1]{fontenc}

\usepackage[a4paper]{geometry}
\usepackage{amsmath,stackengine,mathtools,amssymb}
\usepackage{amsthm} 
\usepackage{comment}
\usepackage{lipsum}
\usepackage[usenames,dvipsnames]{xcolor}
\usepackage{minitoc}
\usepackage{tikz}
\usepackage{filecontents}
\usepackage[numbers,sort&compress]{natbib}
\usepackage[colorlinks=true]{hyperref} 
\hypersetup{urlcolor=blue, citecolor=red} 
\usepackage{tikz} 
\usetikzlibrary{matrix} 
\usepackage{graphicx}
\usepackage[english]{babel}
\usepackage{hyperref}
\usepackage{amsfonts}
\usepackage{empheq} 
\usepackage{tabularx}
\usepackage{appendix}
\usepackage{color} 
\usepackage{lineno} 
\usepackage{float} 
\usepackage{caption} 
\usepackage{verbatim} 
\usepackage{xcolor}   
\usepackage{environ}  
\usepackage{enumitem}       
\usepackage{booktabs}       
\usepackage{diagbox} 
\usepackage{array} 

\usepackage{bigfoot} 
\usepackage[numbered,framed]{matlab-prettifier} 

\usepackage{filecontents} 

	\lstMakeShortInline"
	
\RenewEnviron{comment}{\marginpar{\color{blue}\BODY}}

\newtheorem{thm}{Theorem}[section]

\newtheorem{lemma}{Lemma}[section]
\newtheorem{prop}{Proposition}[section]
\newtheorem{claim}{Claim}[section]
\newtheorem{defin}{Definition}[section]
\newtheorem{oss}{Remark}[section]
\newtheorem{es}{Example}[section]

\newcommand{\numberset}{\mathbb}
\newcommand{\N}{\numberset{N}}
\newcommand{\R}{\numberset{R}}

\newcommand{\nl}[2]{\|#1\|_{L^2(#2)}} 
\newcommand{\nh}[2]{\|#1\|_{H^1(#2)}} 
\newcommand{\h}[1]{\widehat{#1}}

\newcommand{\pt}{\partial_t}
\newcommand{\px}{\partial_x}
\newcommand{\pxx}{\partial^2_x}

\newcommand{\eps}{\varepsilon}

\title{Effect of a membrane on diffusion-driven Turing instability}
\makeatletter
\let\@fnsymbol\@arabic
\makeatother
\author{Giorgia Ciavolella\thanks{Sorbonne Universit{\'e}, Inria, Universit\'{e} de Paris, Laboratoire Jacques-Louis Lions, UMR7598, 75005 Paris, France. Emails: giorgia.ciavolella@sorbonne-universite.fr} \thanks{Dipartimento di Matematica, Università degli Studi di Roma ”Tor Vergata”, Rome, Italy}
}
\date{\today}

\begin{document} 
	
	\maketitle
	
	\begin{center}
		{\small This document is a detailed version of the published paper~\cite{ciavolella} and it represents a chapter of the PhD thesis {\it ”Evolution equations with membrane conditions”} (in preparation).}
	\end{center}

	\begin{abstract} 
	Biological, physical, medical, and numerical applications involving membrane problems on different scales are numerous. We propose an extension of the standard Turing theory to the case of two domains separated by a permeable membrane. To this aim, we study a reaction–diffusion system with zero-flux boundary conditions on the external boundary and Kedem-Katchalsky membrane conditions on the inner membrane. We use the same approach as in the classical Turing analysis but applied to membrane operators. The introduction of a diagonalization theory for compact and self-adjoint membrane operators is needed. Here, Turing instability is proven with the addition of new constraints, due to the presence of membrane permeability coefficients. 
	We perform an explicit one-dimensional analysis of the eigenvalue problem, combined with numerical simulations, to validate the theoretical results. 
	Finally, we observe the formation of discontinuous patterns in a system which combines diffusion and dissipative membrane conditions, varying both diffusion and membrane permeability coefficients.
	The case of a fast reaction-diffusion system is also considered.
	\end{abstract} 
	\vskip .7cm
	
	\noindent{\makebox[1in]\hrulefill}\newline

\section{Introduction}
Pattern formation in a system of reacting substances that possess the ability to diffuse was postulated in $1952$ by Alan Turing~\cite{Turing} and it was numerically studied in $1972$ by Gierer and Meinhardt~\cite{gierer}. A huge literature followed this path in describing animal pigmentation as for the well-studied zebrafish (Watanabe and Kondo~\cite{kondo1}, Yamaguchi \textit{et al.}~\cite{kondo2}), the arrangement of hair and feather in Painter \textit{et al.}~\cite{painter}, the mammalian palate in Economou \textit{et al.}~\cite{kondo4}, teeth in Cho \textit{et al.}~\cite{teeth}, tracheal cartilage rings in Sala \textit{et al.}~\cite{tracheal} and digit pattering in Raspopovic \textit{et al.}~\cite{digit}. In particular, there were found evidences asserting that internal anatomy does not play an influential role in this phenomenon. So, spatial patterns develop autonomously without any pre-pattern structure and they are mathematically described by Turing mechanism. Reaction-diffusion equations are not the only kind of system that exhibits the formation of patterns. Receptor-based models, Klika \textit{et al.}~\cite{klika}, Marciniak-Czochra \textit{et al.}~\cite{marciniak} are an example of organisation mechanisms in a system coupling reaction-diffusion equations and ordinary differential equations. These models are based on the idea that cell differentiate according to positional information. This pre-pattern or morphogen mechanism has been experimentally proven in many morphogenetic events in early development, whereas it is not applicable to the complex structure of the adult body, Kondo \textit{et al.}~\cite{kondo3}. 

Here, we consider another kind of situation which is always a reaction-diffusion system but with a membrane as introduced by Kedem-Katchalsky. In the last twenty years, biological applications of membrane problems have increased. Furthermore, they can describe phenomena on several different scales: from the nucleus membrane, penetrated by molecules such as proteins in the transport between cytoplasm and nucleus (Cangiani and Natalini~\cite{cangiani}, Dimitrio~\cite{dimitrio}, Serafini~\cite{serafini}), to thin interfaces, called basal membranes, degraded by cancer cells with the help of enzymes (Chaplain \textit{et al.}~\cite{giverso}, Ciavolella \textit{et al.}~\cite{ciadapou}, Gallinato \textit{et al.}~\cite{gallinato}, Giverso \textit{et al.}~\cite{giversomulti}), and to exchanges in bloody vessels of blood solutes, such as oxygen, numerically studied in Quarteroni \textit{et al.}~\cite{quarteroni}. Also semi-discretization of mass diffusion problems requires numerical treatment in adjoint domains coupled at the interface (see Calabrò~\cite{calabro}). 

In Ciavolella and Perthame~\cite{ciavper}, the reader can find a previous analytical study on a reaction-diffusion system of $m\geq 2$ species with membrane conditions of the Kedem-Katchalsky type. The main result concerns the existence of a global weak solution in the case of low regularity initial data and at most quadratic non-linearities in an $L^1$-setting. Moreover, it is proven a regularity result such that we have space and time $L^2$ solutions. In particular, solutions are $L^\beta$ in time and $W^{1, \beta}$ in space with $\beta\in [1,2)$, except on the membrane $\Gamma $ where we loose the derivatives regularity. 
So, now the question that arises is whether it is possible to observe patterns in the case species react and diffuse in a domain with an inner membrane and under which conditions.
 
For our purpose, we consider the domain $\Omega = \Omega_l \cup \Omega_r$ with internal interface $\Gamma$ and boundary $\partial \Omega= \Gamma_l \cup \Gamma_r$, where $\Gamma_l:=\partial \Omega_l \setminus \Gamma$, $\Gamma_r:=\partial \Omega_r \setminus \Gamma$. We denote as $\boldsymbol{n}_l$ (respectively, $\boldsymbol{n}_r$) the outward normal to $\Omega_l$ (respectively, $\Omega_r$). We call $\boldsymbol{n}:=\boldsymbol{n}_l=-\boldsymbol{n}_r$.  On the two domains $Q_T^l:=(0,T)\times\Omega_l$ and $Q_T^r:=(0,T)\times\Omega_r$, we consider a reaction-diffusion membrane problem for two species $u$ and $v$ as below. 
\begin{equation}\label{eq}
\left\{
\begin{array}{lll}
\pt u_l-D_{ul}\,\Delta u_l=f(u_l,v_l),\\
& \mbox{in}\; Q_T^l,\\
\pt v_l -D_{v l}\,\Delta v_l=g(u_l,v_l),\\[1ex]
\nabla u_l \cdot n \,=\, 0 \,=\, \nabla v_l\cdot n, & \mbox{in}\; \Sigma_T^l,\\[1ex]
D_{ul}\,\nabla u_l\cdot n=k_u(u_r-u_l),\\
& \mbox{in}\; \Sigma_{T,\Gamma},\\
D_{v l}\,\nabla v_l\cdot n=k_v(v_r-v_l),
\end{array}
\right.\qquad
\left\{
\begin{array}{lll}
\pt u_r-D_{ur}\,\Delta u_r=f(u_r,v_r),\\
& \mbox{in}\; Q_T^r,\\
\pt v_r -D_{v r}\,\Delta v_r=g(u_r,v_r),\\[1ex]
\nabla u_r\cdot n \,=\, 0 \,=\, \nabla v_r\cdot n, & \mbox{in}\; \Sigma_T^r,\\[1ex]
D_{ur}\,\nabla u_r \cdot n=k_u(u_r-u_l),\\
& \mbox{in}\; \Sigma_{T,\Gamma},\\
D_{v r}\,\nabla v_r\cdot n=k_v(v_r-v_l).
\end{array}
\right.
\end{equation}
with $\Sigma_T^l:= (0,T)\times\Gamma_l$, $\;\Sigma_T^r:=(0,T)\times\Gamma_r$ and $\Sigma_{T,\Gamma}:=(0,T)\times\Gamma$.\\

In this chapter, we are interested in the effect of the membrane, represented by the permeability coefficients $k_u, k_v$, for Turing instability to arise under particular conditions on the latter membrane coefficients and on the diffusion ones. With this aim, we extend Turing's theory to the case of membrane operators. 
We recall the definition of a Turing unstable steady state in the case of a linearised system, Murray~\cite{murray}. 
\begin{defin}
	We say that a steady state is Turing unstable for the linearised system if it is stable in the absence of diffusion and unstable introducing diffusion. It is also called diffusion
	driven instability.
\end{defin}
\noindent This is the kind of instability induces spatially structured patterns

As for the standard reaction-diffusion problems, in order to prove Turing instability, we need to introduce a diagonalization theory for compact and self-adjoint membrane operators (see Appendix~\ref{proofdiag}). We introduce the eigenvalue problem of the Laplace operator with Neumann and membrane conditions for each specie $u$ and $v$. We call 
\begin{equation}\label{ll}
   	L=- D_u\Delta \quad \mbox{and} \quad \widetilde{L}=-D_v \Delta,
\end{equation} 
where we define
\begin{equation}\label{shortnotat}
D_\phi= \left\{\begin{array}{ll}
	D_{\phi l}, \mbox{ in } \Omega_l,\\
	D_{\phi r}, \mbox{ in } \Omega_r. 
\end{array}\right. \qquad
\phi=\left\{\begin{array}{ll}
	\phi_{_l},& \mbox{ in } \Omega_l,   \\[1ex]
	\phi_{_r},& \mbox{ in } \Omega_r, 
\end{array}\right.	
\end{equation} 
for $\phi=u$ or $v$.
So, we have for $u$
\begin{equation}\label{eigenw}
\left\{\begin{array}{ll}
L w= \lambda w, & \mbox{ in } \Omega_l \cup \Omega_r, \\[1ex]
\nabla w \cdot n= 0, & \mbox{ in } \Gamma_l \cup \Gamma_r,\\[1ex]
D_{ul} \nabla w_l \cdot n= D_{ur} \nabla w_r \cdot n= k_u(w_r-w_l), & \mbox{ in } \Gamma, 
\end{array}
\right. 
\end{equation}
and for $v$, 
\begin{equation}\label{eigenz}
\left\{\begin{array}{ll}
\widetilde{L} z= \eta z, & \mbox{ in } \Omega_l \cup \Omega_r,\\[1ex]
\nabla z \cdot n= 0, & \mbox{ in } \Gamma_l \cup \Gamma_r,\\[1ex]
D_{vl} \nabla z_l \cdot n= D_{vr} \nabla z_r \cdot n= k_v(z_r-z_l), & \mbox{ in } \Gamma.
\end{array}
\right.
\end{equation}
Thanks to the diagonalization theory introduced in Theorem~\ref{diagon}, we infer the following result.
\begin{prop}
	There exist increasing and diverging sequences of real numbers $\{\lambda_{_{n}}\}_{_{n \in N}}$ and $\{\eta_{_{n}}\}_{n\in\N}$ which are the eigenvalues of $L$ and  $\widetilde{L}$, respectively. We call $ \{ w_{_{n}}\}_{_{n \in N}} $ and $\{z_{_{n}}\}_{n\in \N}$ in $L^2(\Omega_l)\times L^2(\Omega_r)$, the corresponding orthonormal basis of eigenfunctions. 
	In particular, we have that $\lambda_0 =0, w_0 =1/|\Omega|^\frac{1}{2}$ and $\eta_0 =0, z_0=1/|\Omega|^\frac{1}{2}$.
\end{prop}
Finally, we are able to state our main theorem (for more details see Theorem~\ref{thmtur}).
\begin{thm}
	Assume the coefficients of System~\eqref{eq} are such that $w_{_n}=z_{_n}$, for all $n\in \N$. Consider the linearised system around the steady state $(\bar{u},\bar{v})$ with $D_v >0$ fixed and assume appropriate conditions on the linearised reaction terms. 
	Then, for $D_u$ sufficiently small, the steady state $(\bar{u},\bar{v})$ is linearly unstable. Moreover, only a finite number of eigenvalues are unstable. 
\end{thm}

The chapter is organised in four sections and two appendices. In Section~\ref{turing}, we introduce assumptions allowing us to find conditions in order to have Turing instability in the case of a membrane problem. We refer to Theorem~\ref{thmtur} as main result. In Section~\ref{1d}, we restrict the analysis to the one dimensional case, so that we explicit the eigenfunctions and the equations defining the eigenvalues. In Section~\ref{examples}, Turing analysis is completed by some numerical examples performed with a finite difference implicit scheme in Matlab. We investigate in one dimension the effect of the membrane on Turing patterns. In Subsection~\ref{reaction}, we propose our choice of reaction terms and data setting for the numerical examples. In Subsection~\ref{teta}~and~\ref{diffk}, we illustrate some simulations varying respectively the diffusion and the permeability coefficients. In Subsection~\ref{effeps}, thanks to the choice made for the reaction terms, we analyse oscillatory limiting solutions to a fast reaction-diffusion system.  
In Section~\ref{conclu}, a brief conclusion can be found.
At the end of the work, the reader can find two appendices. In Appendix~\ref{proofdiag}, we introduce the diagonalization theorem for compact, self-adjoint membrane operators and we apply it to the operators $L^{-1}$ and $\widetilde{L}^{-1}$. In Appendix~\ref{Tetameth}, we give more details concerning the numerical method behind the simulations presented in Section~\ref{examples} and we provide also the Matlab code.

\section{Conditions for Turing instability}\label{turing}

In order to study Turing instability, we first assume that there exists a homogeneous steady state $(\overline{u},\overline{v})$ which is a non-negative solution of
$$f(\overline{u},\overline{v})=0, \quad g(\overline{u},\overline{v})=0.$$
Then, we analyse its stability for the linearised dynamical system around this steady state. Later, we come back to the linearisation of Equations~\eqref{eq}, {\it i.e.},
\begin{equation}\label{eqlin}
	\left\{
	\begin{array}{lll}
		\pt u_l-D_{ul}\,\Delta u_l=\overline{f}_u u_l+ \overline{f}_v v_l,\\[1ex]
		\pt v_l -D_{v l}\,\Delta v_l=\overline{g}_u u_l+ \overline{g}_v v_l,\\[1ex]
		\nabla u_l \cdot n \,=\, 0 \,=\, \nabla v_l\cdot n, \\[1ex]
		D_{ul}\,\nabla u_l\cdot n=k_u(u_r-u_l),\\[1ex]
		D_{v l}\,\nabla v_l\cdot n=k_v(v_r-v_l),
	\end{array}
	\right.\qquad
	\left\{
	\begin{array}{lll}
		\pt u_r-D_{ur}\,\Delta u_r=\overline{f}_u u_r+ \overline{f}_v v_r,\\[1ex]
		\pt v_r -D_{v r}\,\Delta v_r=\overline{g}_u u_r+ \overline{g}_v v_r,\\[1ex]
		\nabla u_r\cdot n \,=\, 0 \,=\, \nabla v_r\cdot n, \\[1ex]
		D_{ur}\,\nabla u_r \cdot n=k_u(u_r-u_l),\\[1ex]
		D_{v r}\,\nabla v_r\cdot n=k_v(v_r-v_l),
	\end{array}
	\right.
\end{equation}
in which $\overline{f}_u, \overline{f}_v, \overline{g}_u, \overline{g}_v$ are the partial derivatives of the reaction terms evaluated in $(\overline{u},\overline{v})$, 
and we look for conditions such that the previous steady state is unstable. We follow the standard theory in Murray~\cite{murray}, Perthame~\cite{perthame}.\\

{\bf Conditions for the dynamical system to perform a stable steady state}\\
With no spatial variation (eliminating the diffusion term), we can study the stability of the previous steady state applying a linearisation method around $(\overline{u},\overline{v})$, as in \eqref{eqlin}. 
Setting 
$$z=\left(\begin{array}{ll}
u-\overline{u}\\
v-\overline{v}
\end{array}\right),$$
we get 
\begin{equation*}
\pt z = A z, \quad \mbox{ where } A= \left(\begin{matrix}
\overline{f}_u\quad \overline{f}_v\\
\overline{g}_u\quad \overline{g}_v
\end{matrix}\right).
\end{equation*}	
We look for solutions in the exponential form $z\propto e^{\mu t}$, where $\mu$ is the eigenvalue related to the matrix $A$. The steady state $z=0$ is linearly stable if $Re(\mu)<0$. In that case we can observe an exponential decay to zero. This condition is guaranteed if
\begin{equation}\label{trdet}
\mbox{tr}(A)=\overline{f}_u+\overline{g}_v<0 \quad \mbox{ and } \quad \mbox{det}(A)=\overline{f}_u\,\overline{g}_v- \overline{f}_v\, \overline{g}_u>0.
\end{equation} 
In particular, we assume
\begin{equation}\label{eq: activ inhib}
	\overline{f}_u > 0 \quad \text{and} \quad \overline{g}_v <0,
\end{equation} 
{\it i.e.}, $u$ is called activator and $v$ is the inhibitor.\\

{\bf Conditions to obtain an unstable steady state in the case of spatial variation}\\
Now we consider the complete reaction-diffusion systems linearised around the steady state as in \eqref{eqlin}. Referring to the diagonalization theory in Appendix~\ref{proofdiag}, there exist orthonormal basis of eigenfunctions $ \{ w_{_{n}}\}_{_{n \in N}} $  for $L$ and $\{z_{_{n}}\}_{n\in \N}$ for $\widetilde{L}$ in $L^2(\Omega_l)\times L^2(\Omega_r)$.
We use these basis to decompose $u$ and $v$ as
\begin{equation}\label{soluv}
u(t,x)=e^{\mu t} \sum\limits_{n\in\N}\alpha_{_{n}} w_{_{n}}(x), \qquad 	v(t,x)=e^{\mu t} \sum\limits_{n\in\N} \beta_{_{n}} z_{_{n}}(x),
\end{equation}
where $e^{\mu t}\alpha_{_{n}}=( u, w_{_{n}} )_{\bf L^2}$ and $e^{\mu t}\beta_{_{n}}=( v, z_{_{n}} )_{\bf L^2}$, for all $n\in\N$, with ${\bf L^2}$ which is defined as the $L^2$ product space.
\begin{defin}\label{defl2}
	We define ${\bf L^2}=L^2(\Omega_l)\times L^2(\Omega_r)$. We endow it with the norm 
	$$\|w\|_{\bf L^2}= \left(\|w^1\|^2_{L^2(\Omega_l)} +\|w^2\|^2_{L^2(\Omega_r)}\right)^\frac{1}{2}.$$
	We let $(\cdot,\cdot)_{\bf L^2}$ be the inner product in ${\bf L^2}$.
\end{defin}

Substituting \eqref{soluv} into the linearised reaction-diffusion System~\eqref{eqlin} and using \eqref{eigenw} and \eqref{eigenz}, we infer 
\begin{equation}\label{sost}
	\left\{
	\begin{array}{ll}
		\sum_{n}\left(\alpha_{_{n}}\mu w_{_{n}}+ \alpha_{_{n}}\lambda_{_{n}}w_{_{n}}\right) = \sum_{n} \left(\overline{f}_u \alpha_{_{n}} w_{_{n}}+ \overline{f}_v\beta_{_{n}}z_{_{n}}\right),\\[2ex]
		\sum_n\left(\beta_{_{n}}\mu z_{_{n}} +\beta_{_{n}}\eta_{_{n}} z_{_{n}} \right) = \sum_n\left( \overline{g}_u \alpha_{_{n}} w_{_{n}}+ \overline{g}_v\beta_{_{n}} z_{_{n}}\right),
	\end{array}
	\right.
\end{equation}
with boundary conditions well satisfied. Indeed, for $x\in \Gamma$ we deduce that
\begin{equation}\label{kke}
   	\begin{array}{ll}
   		\sum\limits_{n\in\N} (\;\alpha_{_{n}} e^{\mu t} k_u (w_{rn}(x)-w_{ln}(x)) \;) = \sum\limits_{n\in\N} k_u (\;\alpha_{_{n}} e^{\mu t} w_{rn}(x)- \alpha_{_{n}} e^{\mu t} w_{ln}(x)) \;),\\[2ex]
   		\sum\limits_{n\in\N} (\;\beta_{_{n}} e^{\mu t} k_v (z_{rn}(x)-z_{ln}(x)) \;) = \sum\limits_{n\in\N} k_v (\;\beta_{_{n}} e^{\mu t} z_{rn}(x)- \beta_{_{n}} e^{\mu t} z_{ln}(x)) \;),
   	\end{array}
\end{equation}
whereas on the external boundary Neumann conditions are trivial. 
In view of the structure of~\eqref{sost}, it will be convenient, for analysis, to impose $w_{_n}=z_{_n}$, for all $n\in\N$. This is the case under the following conditions.

\begin{lemma}[Conditions  for $w_n=z_n$, for all $n\in\N$]\label{lemmacond}
	Let 
	\begin{equation}\label{nukteta}
		\nu_{_D}:=\frac{D_{ur}}{D_{ul}}=\frac{D_{vr}}{D_{vl}}, \qquad \nu_{_K}:=\frac{k_u}{D_{ul}}=\frac{k_v}{D_{vl}} \qquad \mbox{ and } \qquad \theta:=\frac{D_{ul}}{D_{vl}} =\frac{D_{ur}}{D_{vr}}.
	\end{equation}
	A sufficient condition to have $w_n=z_n$, for all $n\in\N$, is the following relation
	\begin{equation}\label{etalambda}
		\lambda_{_n}=\theta\eta_{_n}, \quad  \mbox{ for all } n\in\N.
	\end{equation}			
\end{lemma}

\begin{proof}
	With relations~\eqref{nukteta}, $w_n$ and $z_n$ solve the same eigenvalue problem (see Problems~\eqref{eigenw}~and~\eqref{eigenz}) for all $n\in\N$.
	From the diagonalization theory (see Theorem~\ref{diagon}), there exists a solutions sequence of eigenvalues and related eigenfunctions. In particular, with condition~\eqref{etalambda}, $w_{_{n}}\propto z_{_{n}}$, {\it i.e.} $w_{_{n}}=C z_{_{n}}$, for all $n\in\N$ but since these basis are orthonormal, the constant $C$ is equal to $1$. 
	
\end{proof}
  
We are now ready to state our main theorem.

\begin{thm}[Turing instability theorem]\label{thmtur}
	Consider the linearised Systems~\eqref{eqlin} around the steady state ($\overline{u}$,$\overline{v}$) with $D_v>0$ fixed.
	We assume \eqref{trdet}-\eqref{eq: activ inhib}, and \eqref{nukteta}-\eqref{etalambda}. Then, for $\theta$ sufficiently small (that means $D_u$), the steady state $(\overline{u}, \overline{v})$ is linearly unstable. Moreover, only a finite number of eigenvalues are unstable. 
\end{thm}

\begin{proof}
Using the orthogonality of the eigenfunctions in Equation~\eqref{sost} and assuming conditions~\eqref{nukteta}~and~\eqref{etalambda} in Lemma~\ref{lemmacond}, we arrive to
\begin{equation}\label{proj}
	\left\{
	\begin{array}{ll}
		\alpha_{_{n}}\mu + \alpha_{_{n}}\lambda_{_{n}}  = \overline{f}_u \alpha_{_{n}}  + \overline{f}_v\beta_{_{n}} ,\\[1ex]
		\beta_{_{n}}\mu   +\beta_{_{n}}\eta_{_{n}}   = \overline{g}_u \alpha_{_{n}}  + \overline{g}_v\beta_{_{n}}.
	\end{array}
	\right.
\end{equation}

This linear system has $\alpha_{_{n}}$ and $\beta_{_{n}}$ as unknowns. In order to have nonnegative solutions we need to assure that the determinant of the coefficients of the system is zero, {\it i.e.} 
\begin{equation*}
\det\left(\begin{matrix}
\mu+\lambda_{_{n}}-\overline{f}_u & -\overline{f}_v\\[1ex]
-\overline{g}_u & \mu+\eta_{_{n}}-\overline{g}_v 
\end{matrix}\right)=0.
\end{equation*}
Hence, we infer that we have the so-called {\it dispersion relation}
\begin{equation}\label{mu}
\mu^2+\mu[\eta_{_{n}}-\overline{g}_v+\lambda_{_{n}}-\overline{f}_u] + \eta_{_{n}}\lambda_{_{n}} - \lambda_{_{n}} \overline{g}_v + \overline{f}_u \eta_{_{n}} + \det(A)=0.
\end{equation}
As underlined in~\eqref{etalambda}, the eigenvalues are proportional. Therefore, through condition~\eqref{nukteta}, we can write that $\lambda_{_{n}}= \theta \, \eta_{_{n}}$.
As a consequence, we can rewrite~\eqref{mu} to have an equation of $\mu(\eta_{_{n}})$. Indeed, we get that 
\begin{equation}\label{mul}
\mu^2+\mu[\eta_{_{n}} (1+\theta)- \mbox{tr}(A)] + \theta \eta_{_{n}}^2 - \eta_{_{n}} (\overline{f}_u + \theta\,\overline{g}_v) + \det(A)=0.
\end{equation}
For the steady state to be unstable to spatial disturbances, we require that $\mbox{Re}(\mu(\eta_{_{n}}))>0$. Since we are working with condition~\eqref{trdet}, the first order coefficient of this polynomial is positive. Consequently, we need to impose that
\begin{equation}\label{nec}
p(\eta_{_{n}}):=\theta \eta_{_{n}}^2 - \eta_{_{n}} (\overline{f}_u + \theta\,\overline{g}_v) + \det(A) <0.
\end{equation}
Because $\eta_{_{n}}$, $\theta$ and $\det(A)$ are positive quantities, the polynomial in \eqref{nec} can take negative values only for 
\begin{equation}\label{condtur}
\overline{f}_u+ \theta \overline{g}_v  >0	
\end{equation}
sufficiently large and $\theta\det(A)$ sufficiently small. We remember that one of the conditions to have stability without diffusion was tr$(A)= \overline{f}_u+\overline{g}_v<0$. This implies that $\theta\neq 1$, in other words $D_{u}\neq D_{v}$.
 
Inequality~\eqref{condtur} is necessary but not sufficient for $\mbox{Re}(\mu(\eta_{_{n}}))>0$. For the convex function $p(\eta_{_{n}})$ to be strictly negative for some nonzero $\eta_{_{n}}$, the minimum must be strictly negative. So if we look for the minimum, we find its coordinates
\begin{equation}\label{min}
	\eta_{\mbox{{\tiny min}}}=\frac{\overline{f}_u+\theta \overline{g}_v}{2\theta} \quad \mbox{ and } \quad p_{\mbox{{\tiny min}}}= \det(A)- \frac{(\overline{f}_u+\theta \overline{g}_v)^2}{4\theta}.
\end{equation}
Then, the condition $p_{\mbox{{\tiny min}}}<0$ corresponds to $\frac{(\overline{f}_u+\theta \overline{g}_v)^2}{4\theta}>\det(A)$.
Finally, given specific functions $f$ and $g$, we can find the values of $\theta$ which assure that the minimum $p_{\mbox{{\tiny min}}}<0$. We call $\theta_c$ the critical diffusion ratio such that $p_{\mbox{{\tiny min}}}=0$, {\it i.e.} the appropriate root of
\begin{equation}\label{tetac}
\overline{g}_v^2\, \theta_c^2 + 2\,(\overline{f}_u \overline{g}_v-2\det(A))\,\theta_c + \overline{f}_u^2=0.
\end{equation}
It corresponds to the value of $\theta$ at which there is a bifurcation phenomenon (see Subsection~\ref{teta} in which we analyse some related examples). 
	
The range of values of $\eta_{_{n}}$ such that $p(\eta_{_{n}})<0$ is $\eta_-<\eta_{_{n}} <\eta_+, \mbox{ with }$
{\small
\begin{equation}\label{range}
\eta_-=\frac{|\overline{f}_u+\theta \overline{g}_v|-\sqrt{|\overline{f}_u+\theta \overline{g}_v|^2-4\theta \det(A)}}{2\theta},\quad
\eta_+=\frac{|\overline{f}_u+\theta \overline{g}_v|+\sqrt{|\overline{f}_u+\theta \overline{g}_v|^2-4\theta \det(A)}}{2\theta}.
\end{equation}}
If we consider the solutions given by \eqref{soluv},  the dominant contribution as $t$ increases are the modes for which Re$(\mu(\eta_{_{n}}))>0$ since all the other modes tend to zero exponentially. By consequence, we can consider the following approximation for large $t$ 
\begin{equation*}
u(t,x)\sim\!\!\!\!\!\sum\limits_{\tiny\begin{array}{cc}
	n{\in}\N\\
	\eta_- {<}\eta_{_n}{<}\eta_+
	\end{array}}\!\!\!\!\!\alpha_{_{n}} e^{\mu(\eta_{_{n}}) t} z_{_{n}}(x) \quad \mbox{ and } \quad v(t,x)\sim\!\!\!\!\!\sum\limits_{\tiny\begin{array}{cc}
	n{\in}\N\\
	\eta_- {<}\eta_{_n}{<}\eta_+
	\end{array}}\!\!\!\!\! \beta_{_{n}} e^{\mu(\eta_{_{n}}) t} z_{_{n}}(x).
\end{equation*}
So, the larger is the range defined by $\eta_-$ and $\eta_+$, the larger is the number of unstable modes not decreasing in time and, then, the modes which infer Turing instability. In order to estimate this interval, we can restrict to the regime $\theta$ small which is the most common in data. In that way, Taylor expansion of the square root gives
\begin{equation*}
\eta_\pm=\frac{\overline{f}_u+\theta \overline{g}_v}{2 \theta} \left[1 \pm \sqrt{1-\frac{4 \det(A) \theta}{(\overline{f}_u+\theta \overline{g}_v)^2}}\right]\; \sim \;\frac{\overline{f}_u}{2\theta } \left[1 \pm \left(1-\frac{2 \det(A) \theta}{(\overline{f}_u+\theta \overline{g}_v)^2}\right)\right].
\end{equation*}
Finally, we obtain
\begin{equation*}
\eta_- \sim \frac{\det(A)}{\overline{f}_u} =O(1) \quad \mbox{ and } \quad \eta_+ \sim \frac{\overline{f}_u}{\theta } \gg 1.	
\end{equation*}
Taking $\theta$ sufficiently small (that means $D_{u}$), the interval $(\eta_-,\eta_+)$ becomes very large, therefore we can find some eigenvalues $\eta_{_{n}}$ in this interval. We remember that $\eta_{_{n}}$ are increasing eigenvalues converging to infinity and so there is only a finite number of them in that interval. This concludes the proof of the theorem.

\end{proof}


\section{One dimensional case}\label{1d}
In the one dimensional case, we can construct an explicit solution of the eigenvalue problem.
We consider the domain $(0,x_m) \cup (x_m,L)$, with $x_m=L/2$. Given relations~\eqref{nukteta}-\eqref{etalambda} and with our short notation \eqref{shortnotat} for $D_v$, the eigenfunctions are determined by
\begin{equation}\label{eigenw1}
\left\{\begin{array}{ll}
-\pxx z_{_{n}}= \frac{\eta_{_n}}{D_v} z_{_{n}},  \\[1ex]
\px z_{_{ln}}(0) = 0 = \px z_{_{rn}}(L), \\[1ex]
\px z_{_{ln}}(x_m) = \nu_{_D} \px z_{_{rn}}(x_m) = \nu_{_K} (z_{_{rn}}(x_m)-z_{_{ln}}(x_m)).
\end{array}
\right. 
\end{equation}	
We decompose $z_{_n}$, for all $n\in\N$, as a combination of sinus and cosinus. Nevertheless, Neumann boundary conditions impose a cosinusoidal form. 
Hence, since eigenfunctions are defined up to a multiplicative constant, we deduce that $z_{_n}$, for all $n\in\N$, has components
$$z_{_{ln}}(x)=  C_1\cos(a_{_n} x) \quad \mbox{ and } \quad z_{_{rn}}(x)= \cos(b_{_n}(x-L)).$$  

\noindent In order to verify Equations~\eqref{eigenw1}, we get, for all $n\in\N$,
\begin{equation*}
a_{_n}^2=\frac{\eta_{_{n}}}{D_{vl}}\quad \mbox{ and }\quad b_{_n}^2=\frac{\eta_{_{n}}}{D_{vr}},
\end{equation*}
so, in particular,
\begin{equation*}
a_{_n}^2=\nu_{_D} b_{_n}^2, \quad \mbox{ with }\;\nu_{_D}=\frac{D_{ur}}{D_{ul}}=\frac{D_{vr}}{D_{vl}}.
\end{equation*}
Since the eigenfunctions satisfy Kedem-Katchalsky membrane conditions, we also have the following conditions on $x_m=L/2$, for all $n\in\N$, 

\begin{subequations}
	\begin{empheq}[left=\empheqlbrace]{align}
	&-C_1\; b_{_n}\sin\left(b_{_n} \sqrt{\nu_{_D}} \;\frac L 2 \right)= \sqrt{\nu_{_D}}\; b_{_{n}} \sin\left(b_{_n} \frac{L}{2} \right),\nonumber\\[1ex]
	&D_{vr}\; b_{_n} \sin\left(b_{_n} \frac{L}{2}\right)= k_v\left(\cos\left(b_{_n} \frac L 2 \right)-C_1 \cos \left(b_{_n}\sqrt{\nu_{_D}}\;\frac{L}{2}\right)\right),\nonumber
	\end{empheq}
\end{subequations}
Then, we infer that, for all $n\in\N$, either $b_{_n}\!=\!0$, so $\eta_{_n}\!=\!0$ and $z_{_{ln}}\!=\!z_{_{rn}}\!=\! \mbox{const}$, or if $b_{_n} \!\neq\! 0$,
\begin{subequations}
	\begin{empheq}[left=\empheqlbrace]{align}
	&C_1= - \sqrt{\nu_{_D}}\; \frac{ \sin\left(b_{_n} \frac{L}{2} \right) }{ \sin\left(b_{_n}\sqrt{\nu_{_D}} \frac{L}{2} \right) },\nonumber\\[1ex]
	&D_{vr}\; b_{_n} \tan\left(b_{_n} \frac{L}{2}\right)=k_v\left[ 1 + \sqrt{\nu_{_D}}\;\frac{ \tan\left(b_{_n} \frac{L}{2}\right) }{ \tan\left(b_{_n} \sqrt{\nu_{_D}} \frac{L}{2}\right) }\right]. \nonumber
	\end{empheq}
\end{subequations}
Hence, we have a system of two equations with $2$ unknowns: $C_1$ and $\eta_{_n}$. 
We conclude that, for all $n\in\N$,
\begin{subequations}
	\begin{empheq}[left=\empheqlbrace]{align}
	&C_1= - \sqrt{\nu_{_D}}\; \frac{ \sin\left(\frac{\sqrt{\eta_{_{n}}}}{ \sqrt{D_{vr}} } \frac{L}{2} \right) }{ \sin\left(\frac{\sqrt{\eta_{_{n}}}}{\sqrt{D_{vl}}} \frac{L}{2} \right) },\\[1ex]
	&\sqrt{\eta_{_{n}}}\; \sqrt{D_{vr}} \tan\left(\frac{\sqrt{\eta_{_{n}}}}{ \sqrt{D_{vr}} } \frac{L}{2}\right)=k_v\left[ 1 + \sqrt{\nu_{_D}}\;\frac{ \tan\left(\frac{\sqrt{\eta_{_{n}}}}{ \sqrt{D_{vr}} } \frac{L}{2}\right) }{ \tan\left(\frac{\sqrt{\eta_{_{n}}}}{ \sqrt{D_{vl}} } \frac{L}{2}\right) }\right].\label{lambda}
	\end{empheq}
\end{subequations}

We can express the eigenvalues as the positive roots of the continuous function $r: \R^+ \rightarrow \R$, such that
\begin{equation}\label{eq: r}
  	r: \quad\xi \quad\longmapsto\quad \sqrt{\xi}\;\; \frac{  \tan\left(\frac{\sqrt{\xi}}{ \sqrt{D_{vl}} } \frac{L}{2}\right) \; \tan\left(\frac{\sqrt{\xi}}{ \sqrt{D_{vr}} } \frac{L}{2}\right) }{ \left[  \tan\left(\frac{\sqrt{\xi}}{ \sqrt{D_{vl}} } \frac{L}{2}\right) + \sqrt{\nu_{_D}}\;\tan\left(\frac{\sqrt{\xi}}{ \sqrt{D_{vr}} } \frac{L}{2}\right) \right] }- \frac{k_v}{\sqrt{D_{vr}} }.
\end{equation}
see Figure~\eqref{lambdageneral}.
\captionsetup[figure]{labelfont=bf,textfont={it}}	
 \begin{figure}[H] 	
    	\begin{center}
    		\includegraphics[scale=0.25]{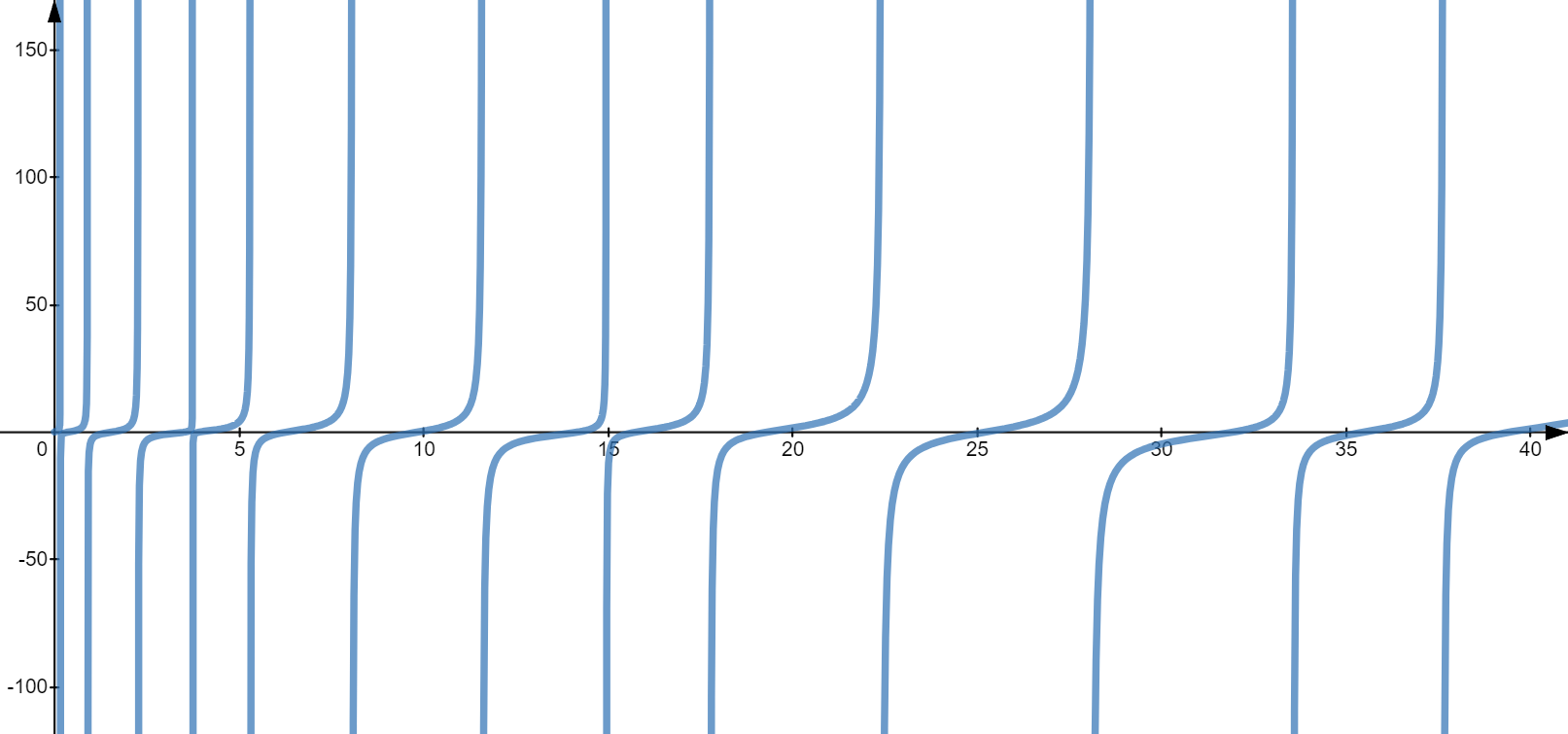}
    		\caption{We represent here the function $\xi \longmapsto r(\xi)$ in \eqref{eq: r}, considering $L=1$, $D_{vl}= 10^{-1}$, $D_{vr}= 10^{-2}$ and $k_v=10^{-4}$. Its roots correspond to the eigenvalues $\eta_{_n}$.}
    		\label{lambdageneral}
    	\end{center}
 \end{figure}

In order to simplify Equation~\eqref{lambda}, in the following, we restrict to the case $\nu_{_D}=1$, {\it i.e.} $D_{vl}=D_{vr}$ and $D_{ul}=D_{ur}$, which is a reasonable assumption when the medium in the left and right domain have similar properties of diffusivity. Then, relation~\eqref{lambda} can be written for all $n\in\N$ as
\begin{equation}
	C_1=-1\quad \mbox{ and } \quad \sqrt{\eta_{_{n}}}\; \tan \left(\frac{\sqrt{\eta_{_{n}}}}{\sqrt{D_{vr}}}\; \frac{L}{2}\right) = 2 \;\frac{k_v}{\sqrt{D_{vr}}}.\label{lambdau}
\end{equation}
The simplified function $r(\cdot)$ of the form 
\begin{equation}\label{eq: r simple}
	r(\xi)= \sqrt{\xi}\;\;  \tan\left(\frac{\sqrt{\xi}}{ \sqrt{D_{vr}} } \frac{L}{2}\right) - 2\frac{k_v}{\sqrt{D_{vr}} },
\end{equation}
is depicted in Figure~\ref{lambdanu1}.

\captionsetup[figure]{labelfont=bf,textfont={it}}
\begin{figure}[H] 	
	\begin{center}
		\includegraphics[scale=0.5]{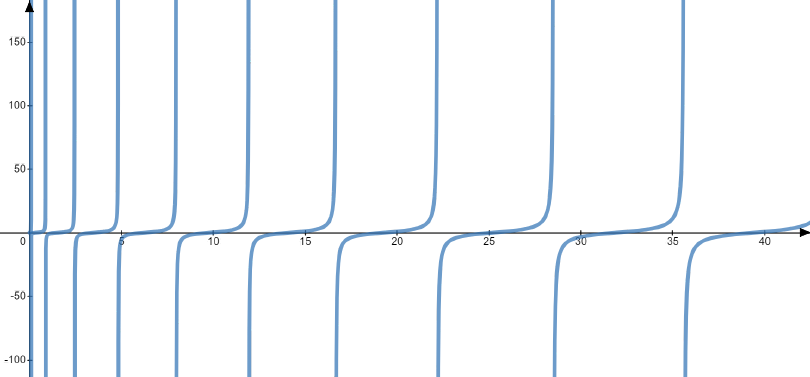}
		\caption{Same as Figure~\ref{lambdageneral}, with $\nu_{_D}=1$, and $D_{vr}= D_{vl}= 10^{-2}$. That is relation~\eqref{eq: r simple}, in place of \eqref{eq: r}.}
		\label{lambdanu1}
	\end{center}
\end{figure}


\section{Numerical examples}\label{examples}
We investigate through numerical examples the effect of the membrane on appearance and shape of Turing's instability. We use the finite difference scheme of a $\Theta$-method with $\Theta=1$, Morton
and Mayers \cite{morton}, Quarteroni et al. \cite{qss}, with a first-order discretization of the boundary and membrane conditions (see Appendix~\ref{Tetameth}). At first, we present in details the expression of the reaction terms and the general data setting that we are using (Subsection~\ref{reaction}). Then, we show some examples. In Subsection~\ref{teta},  we perform numerical examples with different choices for the value of $\theta$ (see Equation~\eqref{nukteta}), referring to the analyses performed in Section~\ref{turing} concerning the values of $\theta_c$ (see Equation~\eqref{tetac}). In Subsection~\ref{diffk}, we exhibit simulations for different values of the membrane permeability coefficients. Finally, in Subsection~\ref{effeps}, we perform oscillatory behaviours when a fast reaction-diffusion system converges to ill-posed cross-diffusion equations and we observe the evolution of these instabilities under the effect of the membrane permeability parameter.

\subsection{Choice of reaction terms and data setting}\label{reaction}
We choose a simple setting with mass conservation, already analysed by Moussa et al. \cite{moussa} in a
Turing instabilities study. In the following, we consider System~\eqref{eq} with 
\begin{equation}\label{react}
f(u,v)= \eps^{-1} (v-h(u)), \quad g(u,v)=-f(u,v), \quad \mbox{ with }\; h(u)~=~\alpha\, u\,(u-~1)^2
\end{equation}
and (see also Figure~\ref{h}) we notice the conditions  
\begin{equation}\label{hprop}
h \in C^2(\R^+,\R^+),\; h(0)=0,\; h(u)>0\; \mbox{ for } u>0\; \mbox{ and }\; h'(u)=\alpha (1-u)(1-3 u) >-1.
\end{equation}
We observe that there is mass conservation which is the first basic property of System~\eqref{eq}~with~\eqref{react}.
Looking at the latter condition $h'(u)\!>\!-1$, the admissible values of $\alpha$ are $0<\alpha<3$. In the numerical examples, we choose the value $\alpha=1$.
The small parameter $\eps>0$ measures the time scale of the reaction compared to diffusion. The smaller is $\eps$, the more numerous are the patterns. Indeed, for $\eps < 1$, we are dealing with a fast reaction-diffusion system and, in the limit $\eps\rightarrow 0$, its Turing instability turns out to be equivalent to the instability due to the ill-posedness for the limiting cross-diffusion equations, caused by backward parabolicity, Moussa \textit{et al.}~\cite{moussa}, Perthame and Skrzeczkowski~\cite{jakubperth}.
In the following numerical examples, we take $\eps=1$ which corresponds to a standard reaction-diffusion system, whereas in Subsection~\ref{effeps} we let  vary $\eps$ to obtain the numerical zero-limit. 

We briefly prove that the reaction terms in \eqref{react}, with general values of $\eps, \alpha$ and $h$, satisfy the analysis in Section~\ref{turing}.
\begin{claim}\label{claim}
	Considering reaction terms in \eqref{react}, we claim that:
	\begin{enumerate}
		\item In the absence of diffusion, there is a unique stable equilibrium point $(\overline{u},\overline{v})$ to which
		solutions converge monotonically.
		\item The same steady state $(\overline{u}, \overline{v})$ is asymptotically Turing unstable for the linearised reaction-diffusion system under the condition \begin{equation}\label{condturh}
			\theta\!+\!h'(\overline{u})\!<\!0.
		\end{equation}
	\end{enumerate}
\end{claim}

\begin{proof}
{\it Statement $1.$}\\
We take the dynamical system
$$\frac{d}{dt} \left(\begin{array}{ll}
	u\\
	v
\end{array}\right) = \left(\begin{array}{ll}
	\eps^{-1} (v-h(u))\\
	-\eps^{-1} (v-h(u))
\end{array}\right),$$
which has steady state $(\overline{u}, \overline{v})$ such that $\overline{v}=h(\overline{u})$. 
Thanks to mass conservation of the system, we can write $M:=u(t)+v(t)=u(0)+v(0)$ and $\frac{d}{dt} u= \eps^{-1} (M-u-h(u))=:\eps^{-1} G(u(t))$. Since $u,v$ are positive functions, the function $G$ has the following properties: $G(0)=M>0$, $G'(u)<0$ and $G(+\infty)=-\infty$. Consequently, there exists a unique stable equilibrium point $(\overline{u}, \overline{v})$, monotonically achieved (since $G(u)>0$ for $u\leq \overline{u}$ and $G(u)<0$ for $u\geq \overline{u}$), that cancels $G$ such that $\overline{u}=M-\overline{v}$ and $\overline{v}=h(\overline{u})$.

\noindent{\it Statement $2.$}\\ 
Applying the same general steps as in the proof of Theorem~\ref{thmtur}, for the steady state to be unstable under spatial disturbances we require (see \eqref{condtur}) that $\theta + h'(\overline{u})<0,$
with $-\eps^{-1} (\theta + h'(\overline{u}))$ sufficiently large and $\theta$ sufficiently small.
This is a necessary and sufficient condition when it is assured that the minimum of the polynomial $p(\eta)$ in \eqref{nec} is negative. Looking back at Equations~\eqref{min} with reactions in \eqref{react}, we get
\begin{equation}\label{minfg}
	\eta_{\mbox{{\tiny min}}}=\eps^{-1}\,\frac{|h'(\overline{u})+\theta|}{2\theta} \quad \mbox{ and } \quad p_{\mbox{{\tiny min}}}=-\theta\eta_{\mbox{{\tiny min}}}^2.
\end{equation}
It is clear that $p_{\mbox{{\tiny min}}}<0$ for all $\eta_{\mbox{{\tiny min}}}\neq 0$, {\it i.e.} for $\theta \neq -h'(\overline{u})$. Otherwise, $p_{\mbox{{\tiny min}}}$ is equal to zero and, then, we have found the critical diffusion ratio $\theta_c = -h'(\overline{u})$ at which there is a bifurcation phenomenon.
Moreover, calculating the range where we can find unstable modes, like in \eqref{range}, we deduce that 
\begin{equation}\label{rangefg}
	\eta_-=0 \; \mbox{ and } \eta_+=-\eps^{-1}\left(1+\theta^{-1} h'(\overline{u})\right).
\end{equation}
This range is larger if condition \eqref{condturh} with $-\eps^{-1} (\theta + h'(\overline{u}))$ sufficiently large and $\theta$ sufficiently small are satisfied. 
In particular, varying the parameter $\eps$, we observe that the smaller it is, the larger is the range $(\eta_-,\eta_+)$, {\it i.e.} a larger number of eigenvalues generating instability can be found. This concludes the proof of the claim.

\end{proof}

We can easily calculate the steady state $(\overline{u},\overline{v})$ thanks to the mass conservative structure of the system, as pointed out in the previous proof. Indeed, adding up the reaction-diffusion equations for $u$ and $v$ and integrating over the space, we get for all $t\geq 0$,
\begin{equation*}
	\int_{0}^{L} u(x,t)+v(x,t) dx= \int_{0}^{L} u_0(x)+v_0(x) dx.
\end{equation*}
Then, we conclude that the steady state depends on the length of the domain ( here $[0,L]$ ) and on the initial data, {\it i.e.}
\begin{equation}\label{equil}
	\overline{u}+\overline{v}= \frac{1}{L} \; \int_{0}^{L} u_0(x)+v_0(x) dx, \quad \mbox{ with } \overline{v}=h(\overline{u}).
\end{equation}
In particular, this steady state is Turing unstable when $h'(\overline{u})< - \theta$, as it can be deduced from relation~\eqref{condturh}. So, $h'(\overline{u})<0$ which means that $h'(\overline{u})\in \left(-\frac{\alpha}{3}, 0\right)$. Then, we infer that the Turing unstable steady state is such that $\overline{u}\in \left(\frac 1 3, 1\right) $ and $\overline{v}\in \left(0, \alpha\frac{4}{27}\right) $ (see Figure~\ref{h}).

\captionsetup[figure]{labelfont=bf,textfont={it}}
\begin{figure}[H] 	
	\begin{center}
		\includegraphics[scale=0.17]{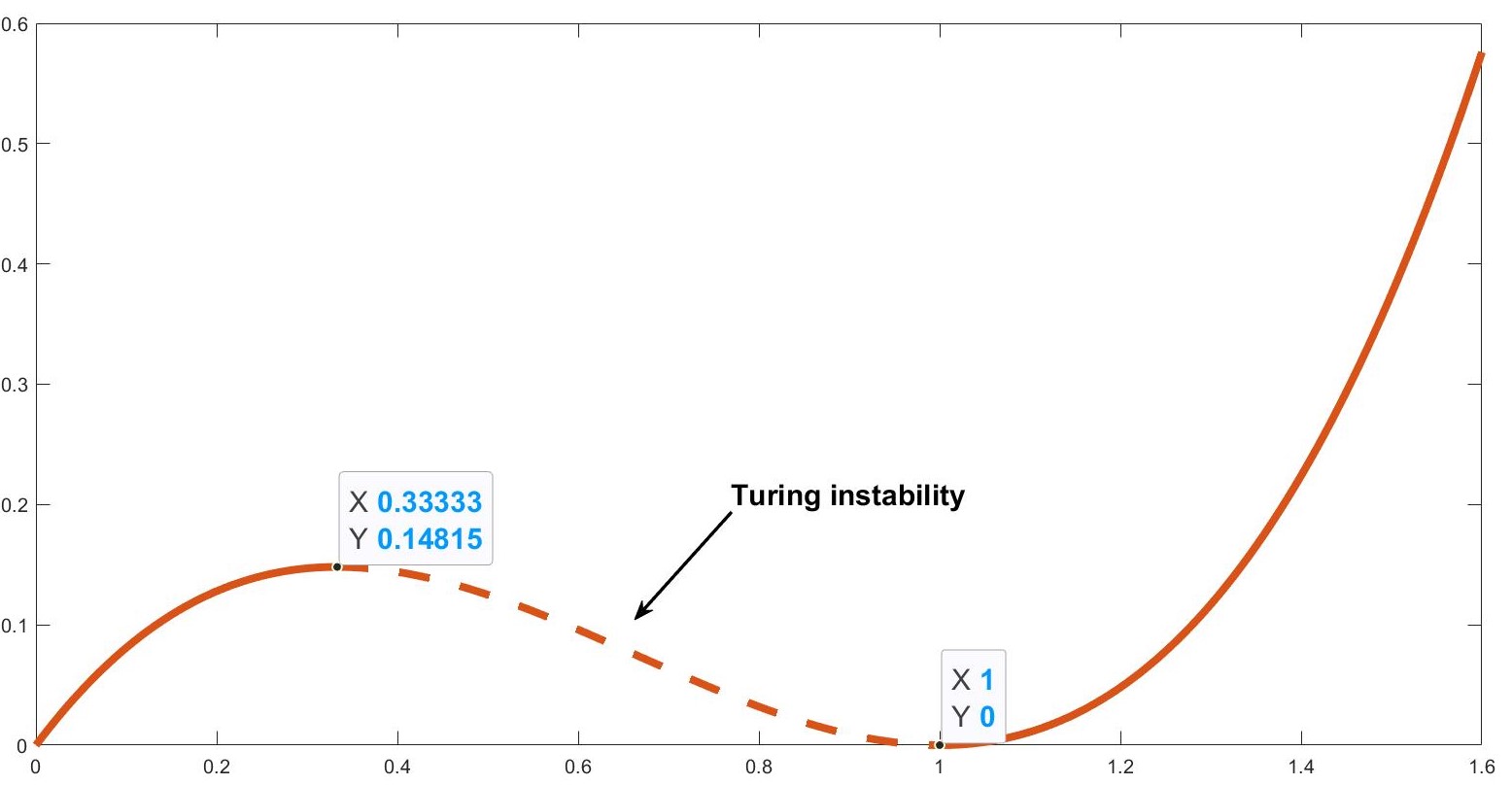}\quad \includegraphics[scale=0.17]{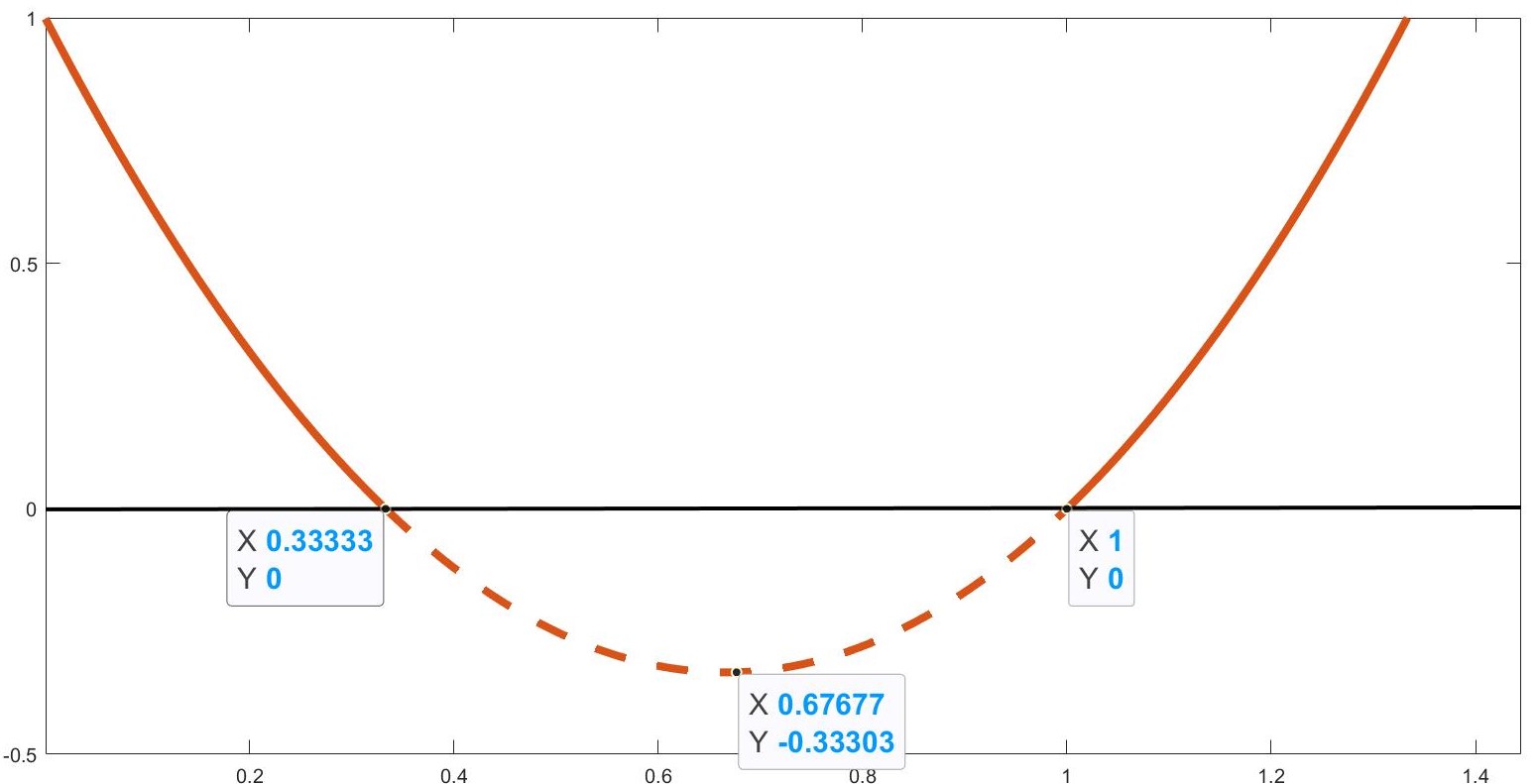}
		\caption{We represent $h(u)$ in \eqref{react}-\eqref{hprop} (left picture) and $h'(u)$ (right picture) with $\alpha=1$. In dashed lines, the instability region for $\overline{u}$ and $h(\overline{u})=\overline{v}$.}
		\label{h}
	\end{center}
\end{figure}

Finally, we present the main data chosen for simulations in Subsection~\ref{teta}~-~\ref{effeps}. We show the time convergent solutions (in the left for $u$ and in the right for $v$) in the spatial interval $[0,L]$, with $L=1$ and with a discretization step $\Delta x=\frac{L}{200}$. 
As shown in Figure~\ref{initdata}, we take the initial data as
\begin{equation*}
	u_0(x)=\left\{
	\begin{array}{ll}
		\frac{7}{15}+\frac 1 5\sin(4\pi x), & \mbox{for } 0\leq x\leq \frac{1}{2},\\[1ex]
		\frac 1 5+\frac 1 5 \sin(4\pi x), & \mbox{for } \frac 1 2 < x\leq 1
	\end{array}\right.
	\mbox{ and }\;
	v_0(x)=\left\{
	\begin{array}{ll}
		\frac 1 3-\frac 1 5\sin(4\pi x), & \mbox{for } 0\leq x\leq \frac{1}{2},\\[1ex]
		\frac 3 5 - \frac 1 5 \sin(4\pi x), & \mbox{for } \frac 1 2 < x\leq 1.
	\end{array}\right.
\end{equation*}
\captionsetup[figure]{labelfont=bf,textfont={it}}
\begin{figure}[H] 	
	\begin{center}
		\includegraphics[scale=0.3]{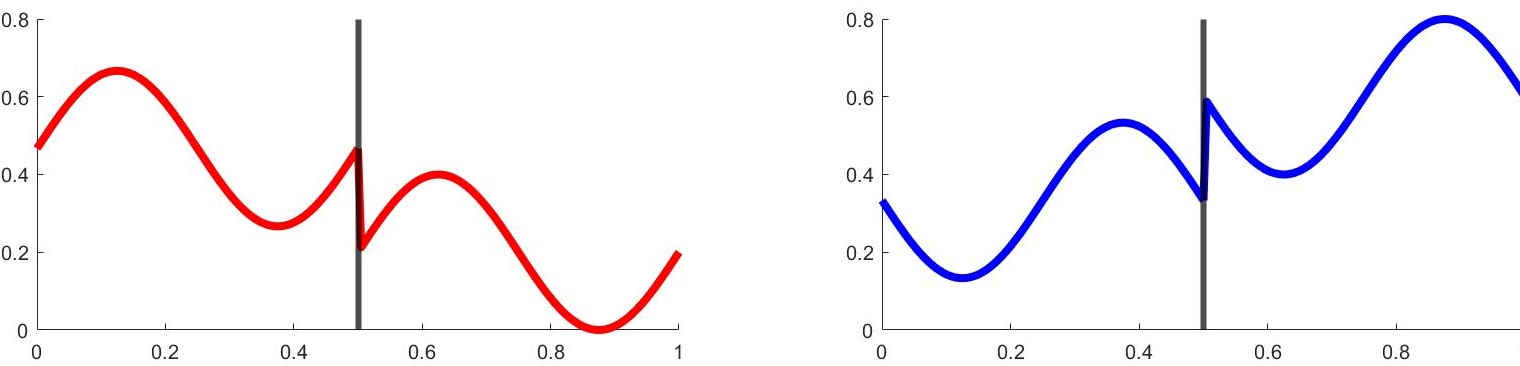}
		\caption{Representation of the initial data $u_0$ (in the left) and $v_0$ (in the right). }
		\label{initdata}
	\end{center}
\end{figure}
\noindent Looking back at \eqref{equil}, we deduce that the steady state $(\overline{u}, \overline{v})$ is such that
$\overline{u}+\overline{v}=\frac{4}{5}$ with $\overline{v}=h(\overline{u})$. With $\alpha=1$, we conclude that 
\begin{equation}\label{equileh}
	\overline{u}=0.7545\in \left(\frac 1 3, 1\right), \;\;\overline{v}=h(\overline{u})=0.0454 \in \left(0,\frac{4}{27}\right) \;\mbox{ and }\; h'(\overline{u})=-0,3101\in \left(-\frac 1 3, 0\right).
\end{equation}
If not specified, we guarantee conditions \eqref{nukteta} and \eqref{etalambda} in Lemma~\ref{lemmacond} with $\nu_{_D}=1$ such that
\begin{equation}\label{param}
	D_{ul}=D_{ur}=\theta,\quad  k_u= \theta \, k_v \quad\mbox{ with } D_{vl}=D_{vr}=1\; \mbox{ and }\; \eps=1.
\end{equation}

\subsection{Effect of the diffusion ratio}\label{teta}
	We illustrate the effect of different values of the diffusion ratio~$\theta$ in~\eqref{nukteta}. We consider the reaction terms in \eqref{react}, initial data as in Figure~\ref{initdata} and data setting as in \eqref{param} 
	with $k_v=1$ fixed. We remember that when we vary $\theta$, there exists a critical diffusion ratio $\theta_c$ for Turing's instability. As analysed in the proof of Claim~\ref{claim} and in \eqref{minfg}, we can define
	\begin{equation}
		\theta_c= -\, h'(\overline{u}) \quad \mbox{ and } \quad \eta_{_{min}}= \frac{1}{2\theta\,\eps}\; |\theta + h'(\overline{u})|, \quad p_{_{min}}= -\, \theta\, \eta_{_{min}}^2,
	\end{equation}
	where $\theta_c$ is the critical diffusion ratio at which $p_{_{min}}$, the minimum of the polynomial \eqref{nec} calculated in $\eta_{_{min}}$, is zero.
	For $\theta= \theta_c$, we remark that $\eta_{_{min}}= p_{_{min}}=0$. Otherwise, for $\theta< \theta_c$, the minimum is strictly negative (see Figure~\ref{p}) and so we can calculate the non-empty range of instability. However, in the case $\theta >\theta_c$, {\it i.e.} $\theta > |h'(\overline{u})|$, we cannot find Turing patterns, since condition \eqref{condturh} does not hold.
	\captionsetup[figure]{labelfont=bf,textfont={it}}
	\begin{figure}[H] 	
		\begin{center}
			\includegraphics[scale=0.26]{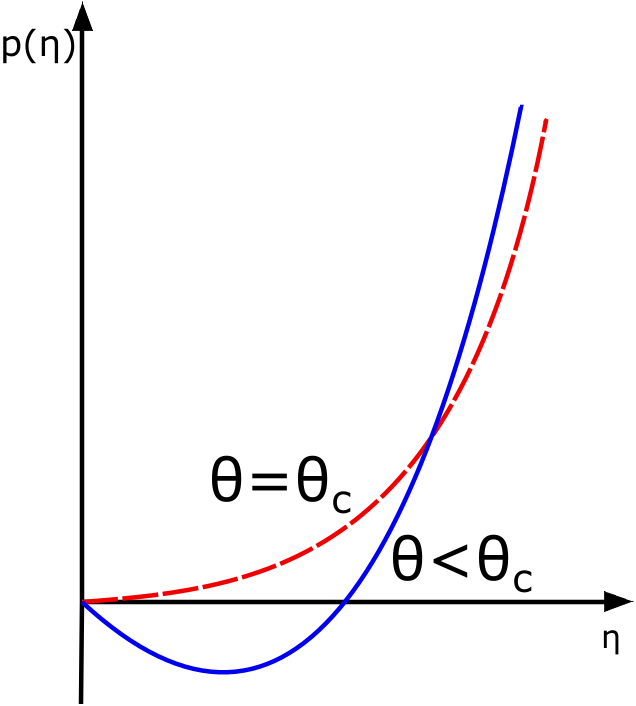}
			\caption{Representation of the function in \eqref{nec} determining the unstable modes with the reaction terms in \eqref{react} for $\theta=\theta_c$ (dashed line) and for $\theta < \theta_c$ (solid line). So, $p(\eta)=\theta \eta^2+\eps^{-1} (\theta + h'(\overline{u})) \eta$ with $\eps^{-1}=2$, $h'(\overline{u})=-0.3101$ and $\theta=10^{-1}<\theta_c$. }
			\label{p}
		\end{center}
	\end{figure}
	
	In the numerical examples, we consider decreasing values of $\theta\leq \theta_c$ in order to see both what happens in an appropriate neighbourhood of $\theta_c$ and far away from this threshold. Looking back at \eqref{equileh}, we infer that $\theta_c= -h'(\overline{u}) = 3.101 \cdot 10^{-1}$.  We recall the expression for $\eta_-, \eta_+$ in \eqref{rangefg} and the one dimension Equation~\eqref{lambdau} that defines the eigenvalues of $u$ and $v$:
	$$\eta_-=0, \quad \eta_+ = -\eps^{-1} (1 + \theta^{-1} h'(\overline{u})) \quad\mbox{ and } \quad\sqrt{\eta_n}\; \tan \left(\frac{\sqrt{\eta_n}}{\sqrt{D_{vr}}}\;\frac{L}{2}\right) = 2 \;\frac{k_v}{\sqrt{D_{vr}}}.$$
	
	\begin{description}
		\item[Case 1.] We take $\theta= \theta_c=  3.101 \cdot 10^{-1}$ and the other parameters according to \eqref{param} ($k_v=1$, $k_u=3.101\cdot 10^{-1}$). 
		\captionsetup[figure]{labelfont=bf,textfont={it}}
		\begin{figure}[H] 	
			\begin{center}
				\includegraphics[scale=0.35]{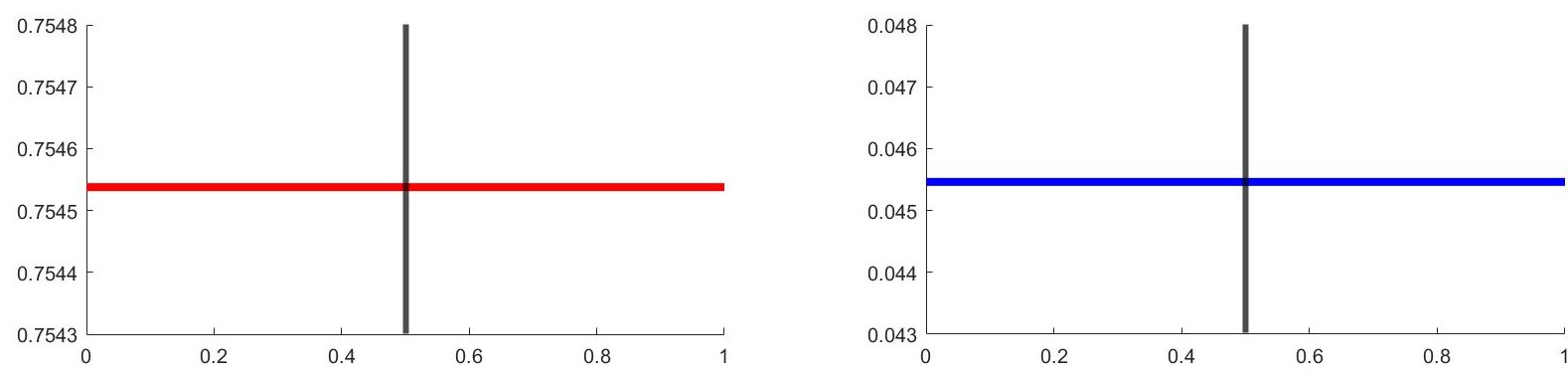}
				\caption{Taking $\theta= \theta_c$, as we can see in Figure~\ref{p}, we cannot define an unstable range $(\eta_-, \eta_+)$ such that the polynomial $p(\eta)$ is strictly negative. In fact, we are at the bifurcation point. That is why, on a long time scale, we do not observe patterns neither for $u$ (in the left) nor for $v$ (in the right). Instead, as we are working with a reaction-diffusion equation with dissipative membrane conditions, we notice the convergence to the equilibrium $(\overline{u},\overline{v})$ in \eqref{equileh}.}
				\label{1tetac}
			\end{center}
		\end{figure}
		
		\item[Case 2.] We take $\theta\!=\! 7.8\cdot 10^{-2}$ and the other parameters according to \eqref{param} ($k_v\!=\!1$, $k_u\!=\!7.8\cdot~10^{-2}$). In this case, $\eta_+ = 2.97$ and so only the first eigenvalue $\eta_1= 2.96$ corresponds to an unstable mode ($\eta_n > \eta_+,$ for $n\geq 2$). 
		\captionsetup[figure]{labelfont=bf,textfont={it}}
		\begin{figure}[H] 	
			\begin{center}
				\includegraphics[scale=0.35]{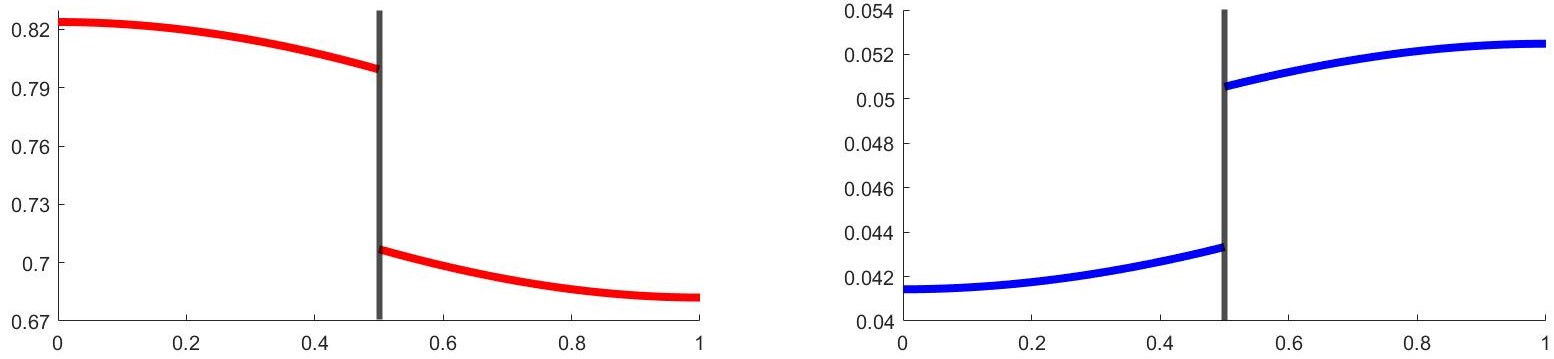}
				\caption{Since $\theta <\theta_c$, on a long time scale, solutions do not reach the steady state even if they are nearby. Considering the only $\eta_1\in (\eta_-, \eta_+)$, we do not observe a really interesting pattern but a piecewise function. We can appreciate the inclination of the solutions in the left and right limit at the membrane: they satisfy Kedem-Katchalsky conditions. We remark that with membrane problems, a nearly constant function with a jump at the membrane stands for a pattern.}
				\label{2teta}
			\end{center}
		\end{figure}
		
  \item[Case 3.] We consider $\theta\!=\! 3\cdot 10^{-4}$ and the other parameters according to \eqref{param} ($k_v\!=\!1$, $k_u\!=\!3\cdot~10^{-4}$). These data give $\eta_+ = 1032.6$ and so we have $6$ eigenvalues in $(\eta_-,\eta_+)$.
  \captionsetup[figure]{labelfont=bf,textfont={it}}
  \begin{figure}[H] 	
	    	\begin{center}
	    		\includegraphics[scale=0.35]{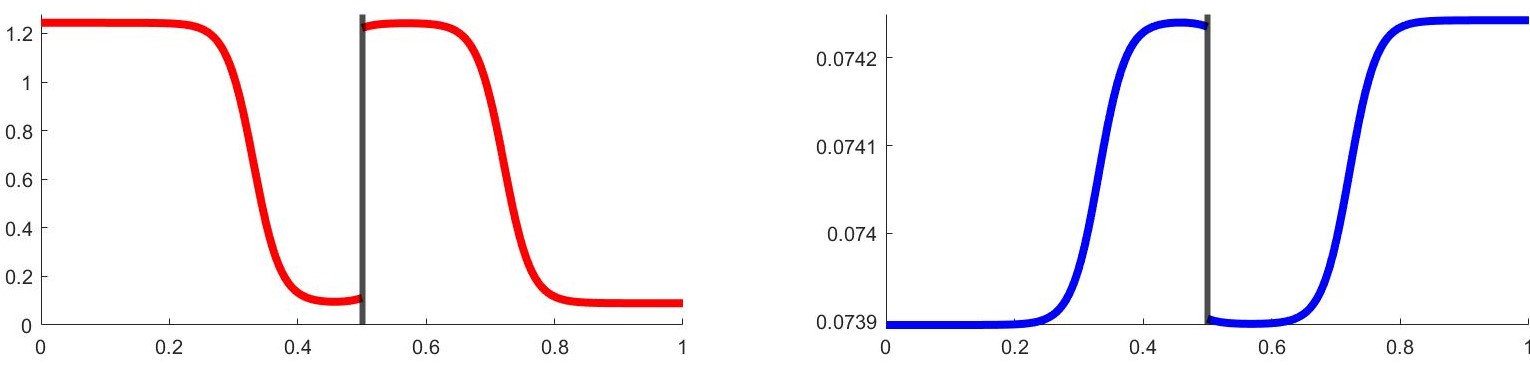}
	    		\caption{Choosing $\theta=3\cdot 10^{-4}$, we succeed in having more considerable patterns for both the species $u$ and $v$ in the temporal limit.
	    		Moreover, it is again clear the well-verification of membrane conditions. As remark, we underline that until $5$ eigenvalues in $(\eta_-,\eta_+)$, over long time interval, the shape does not change significantly respect to Figure~\ref{2teta}. Then, the diffusion ratio $\theta$ has to be sufficiently small to appreciate more complex patterns. }
	    		\label{3teta}
	    	\end{center}
  \end{figure}
					
		\item[Case 4.] We take $\theta= 10^{-5}$ and the other parameters according to \eqref{param} ($k_v=1$, $k_u= 10^{-5}$). In this case, $\eta_+ = 31009$ and so we have several eigenvalues in $(\eta_-,\eta_+)$.
		\captionsetup[figure]{labelfont=bf,textfont={it}}
		\begin{figure}[H] 	
			\begin{center}
				\includegraphics[scale=0.3]{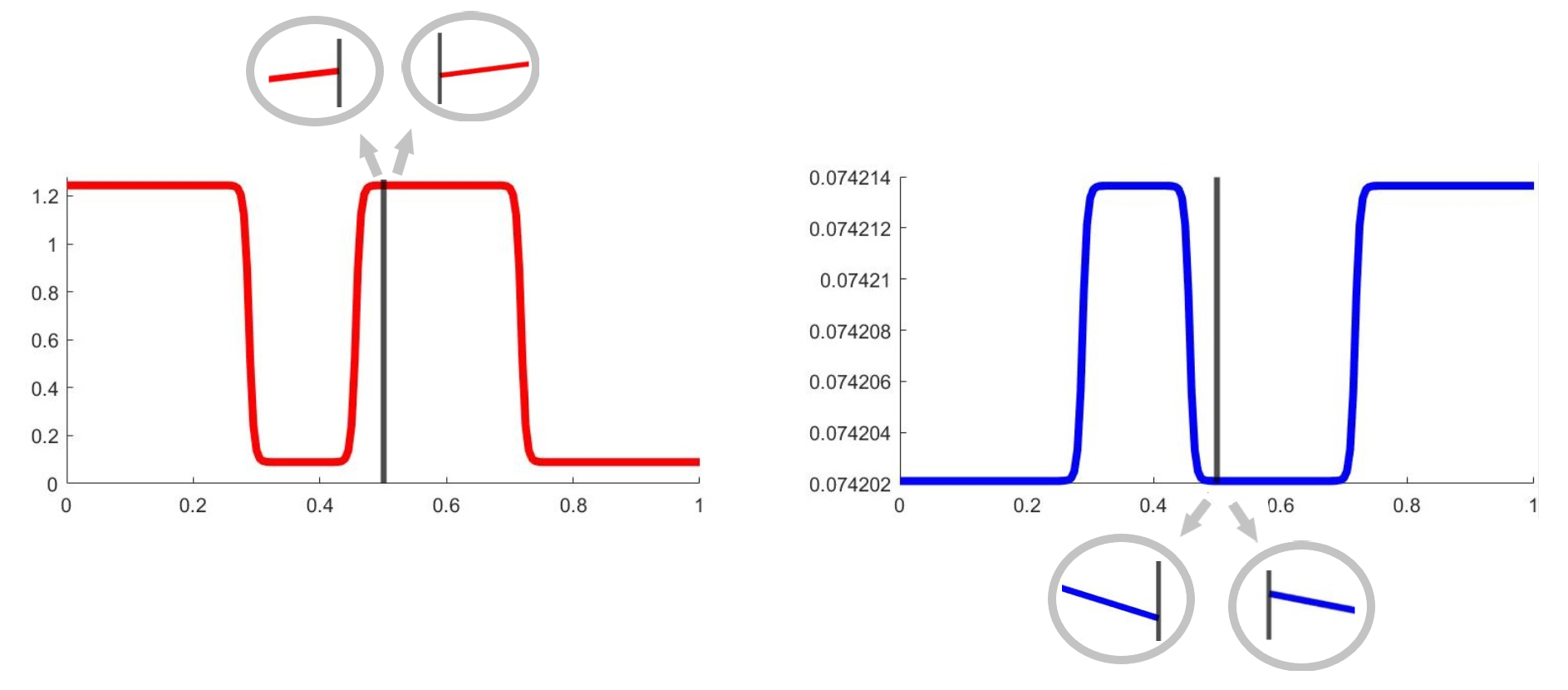}
				\caption{Here, $\theta$ is on a very different scale respect to $\theta_c$ and there is a big number of unstable modes $\eta_n$. Hence, we observe remarkable and beautiful patterns both for $u$ and $v$. The jump at the membrane is not evident with this choice of parameter. Then, in the zoom circles, we can appreciate the inclination of the solutions in the left and right limit at the membrane remarking that they satisfy Kedem-Katchalsky conditions.}
				\label{4teta}
			\end{center}
		\end{figure}
\end{description}		
\vspace{-1mm}
In conclusion, fixing $k_v\in(0,+\infty)$ and decreasing $\theta$ from its critical value $\theta_c$, we can notice a remarkable change in patterns. In particular, starting from the convergence to the equilibrium for $\theta=\theta_c$ in Figure~\ref{1tetac}, we then approach three different, but discontinuous, shapes. Considering a reduced number of eigenvalues in the unstable range, 
solutions show a basic pattern which is a nearly constant function with a jump at the membrane (as in Figure~\ref{2teta}). Decreasing $\theta$, we get more complex and stiffer shapes depending on the number of unstable modes found in the interval $(\eta_-,\eta_+)$ (see Figures~\ref{3teta}~,~\ref{4teta}).

\subsection{Values of the permeability coefficients}\label{diffk}

We show here another set of simulations in which we vary only the permeability coefficient $k_v\in[0,+\infty]$ (then, $k_u$, given the coupling $k_u=\theta k_v$ deducible from~\eqref{nukteta}) in the data chosen in \eqref{param}. So, we better discover the effect of the membrane on Turing patterns. In particular, we can distinguish two limiting situations: $k_v=~0=~k_u$, which is the one without transmission and it corresponds  to have two separate and not communicating domains, and $k_v=+\infty=k_u$ (numerically realised taking $k_v= 10^8$), {\it i.e.} we have full permeability at the membrane, so it corresponds to have a unique connected domain. In this two extreme cases, we recover the results of a standard reaction-diffusion system without the effect of the membrane. 
Considering different values of the permeability coefficients, we can estimate the position of the eigenvalues on the real lines and then, in the unstable interval, in order to follow the same arguments as in the previous subsection.
Indeed, we recall the dependence on $k_v$ of the eigenvalues equation~\eqref{lambdau} such that if $\eta_{_n}\neq 0$, we have that
$$\sqrt{\eta_{_n}}\; \tan \left(\frac{\sqrt{\eta_{{_n}}}}{\sqrt{D_{vr}}}\;\frac{L}{2}\right) = 2 \;\frac{k_v}{\sqrt{D_{vr}}}.$$	
{\bf In the case } $\boldsymbol{k_v=0=k_u}$, the previous equation reduces to $\sin\left(\frac L 2 \frac{\sqrt{\eta_{{_n}}}}{\sqrt{D_{vr}}}\right)=0$ and so we can calculate the eigenvalues as 
$$\eta_{{_n}}= D_{vr}\;\frac{(2n)^2 \pi^2}{L^2}.$$
{\bf In the case }  $\boldsymbol{k_v=+\infty=k_u}$, we have $\cos\left(\frac L 2 \frac{\sqrt{\eta_{{_n}}}}{\sqrt{D_{vr}}}\right)=0$ and, then, the eigenvalues are of the form
$$\eta_{{_n}} = D_{vr}\; \frac{(2n+1)^2 \pi^2}{L^2}.$$

We can affirm that the eigenvalues $\eta_n^k$ related to a certain value of $k=k_u, k_v\in(0,+\infty)$ are situated between the eigenvalues $\eta_n^0$ for $k=0$ and the ones for $k=+\infty$, {\it i.e.} $\eta_n^{\infty}$. Moreover, fixing $n$ and varying $k$, the eigenvalues $\eta_n^k$ pass continuously from $\eta_n^0$ to $\eta_n^{\infty}$. This can be observed in two different ways: from a numerical result or a more analytical one.\\

{\it  Numerical result}\\
For $L=1$, we consider the continuous function
\begin{equation}\label{xi}
	q:\quad\xi \quad \longmapsto \quad \xi \;\tan\left(\frac \xi 2 \right)-2\;\frac{k_v}{D_{vr}}.
\end{equation}
Numerically, we find the zeros $\xi_n= \frac{\sqrt{\eta_{_n}}}{\sqrt{D_{vr}}},\; n\geq 0$ for different values of $\frac{k_v}{D_{vr}}$ and, then, of $k_v$ (see Table \ref{tabeigenu}).
\captionsetup[table]{labelfont=bf,textfont={it}} 
\begin{table}[H]
	\centering
	\begin{tabular}{|c|c|c|c|c|}
		\hline
		$k_v/D_{vr}$     & $\xi_1$     & $\xi_2$     & $\xi_3$       & $\xi_4$\\
		\hline
		$0$              & $0$         & $2\pi$      & $4\pi$        & $6\pi$ \\[1ex]
		$0.5 $           & $0.41\pi$   & $2.09\pi$   & $4.05\pi$     & $6.04\pi$  \\[1ex]
		$ 5 $            & $0.83\pi$   & $2.56\pi$   & $4.39\pi $    & $6.29\pi$  \\[1ex]
		$10^8$           & $\pi$       & $3\pi$      & $5\pi $       & $7\pi$  \\[1ex]
		\hline
	\end{tabular}
	\caption{We report the values of the first four zeroes $\xi_n = \frac{\sqrt{\eta_{_n}}}{\sqrt{D_{vr}}}, \; n=1,...,4$ for different values of $k_v$, since $D_{vr}$ is fixed, including the two limiting cases and two intermediate ones.}
	\label{tabeigenu}
\end{table}
\noindent Then, we recover the previous eigenvalue formulas for the two limiting situations and we can also observe that for fixed $n$, the eigenvalues $\eta_n^{k_v}$ increase continuously with $k_v$ towards $\eta_n^\infty$.\\

{\it Analytical result}\\	
Another way to look at this phenomenon and to better observe continuity of the $\xi_n$'s changing $k_v$ and fixing $n$, it is to represent the function in \eqref{xi} (see Figure~\ref{tang}). We consider $n=1$ and so the interval $\xi \in (0,\pi)$. Since we have a monotonous function for $\xi \in (0,\pi)$, there exists a unique intersection with the horizontal line $y=k:=2\frac{k_v}{D_{vr}}=2\frac{k_u}{D_{ur}}$ and for $0<k_1<k_2<+\infty$, we get $0<\xi_1^{k_1}<\xi_1^{k_2}<+\infty$.
\captionsetup[figure]{labelfont=bf,textfont={it}}
\begin{figure}[H] 	
	\begin{center}
		\includegraphics[scale=0.3]{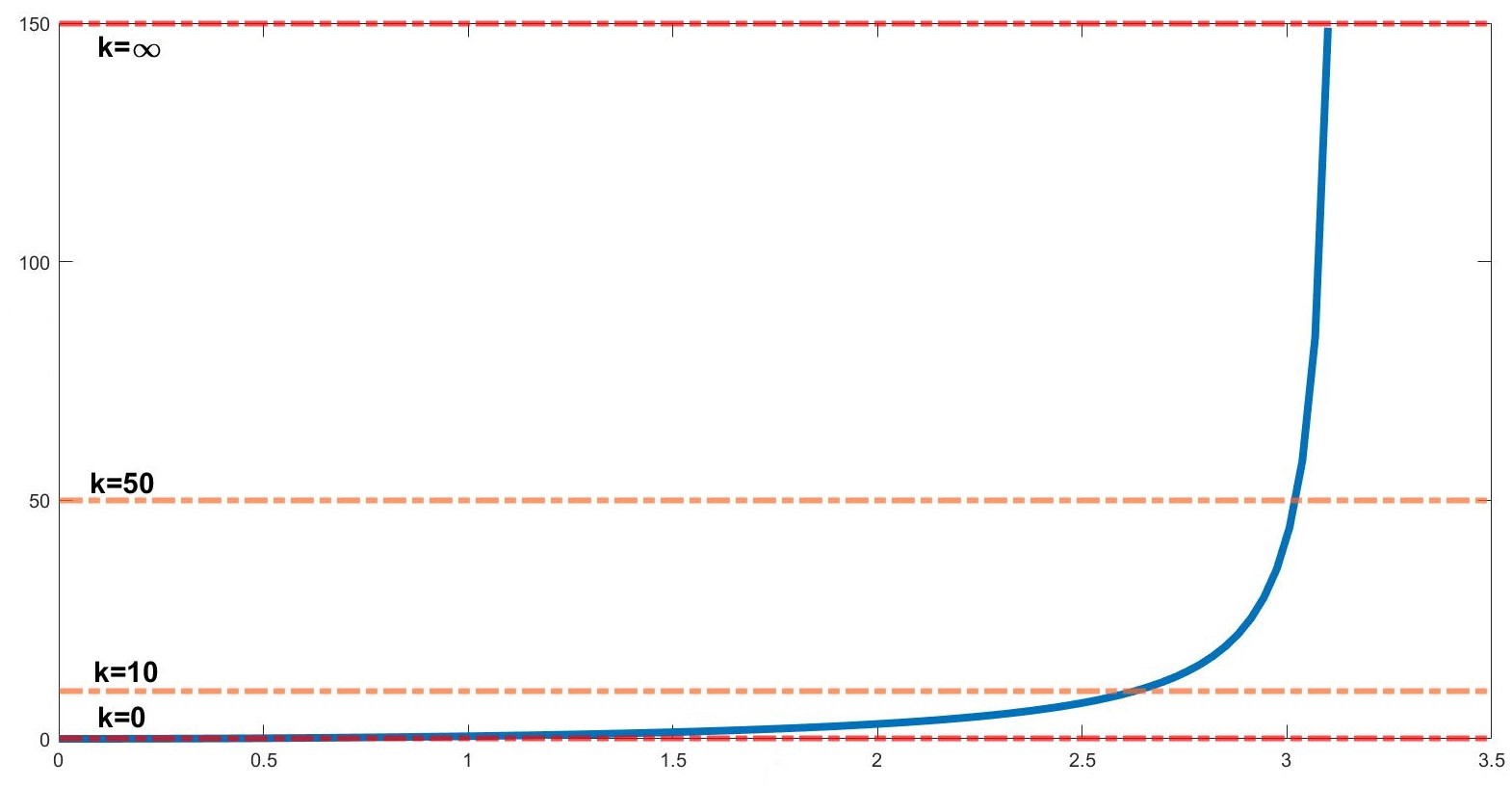}
		\caption{Representation of the first root $\xi_1$ of $q$ in \eqref{xi}, as the intersection between the function $\tilde q:\;\xi\longmapsto\xi\tan\left(\frac \xi 2\right)$ for $\xi \in (0,\pi)$ (solid line) and one of the dashed lines defined by the permeability coefficient by the relation $k:=2\frac{k_v}{D_{vr}}=2\frac{k_u}{D_{ur}}$.}
		\label{tang}
	\end{center}
\end{figure}

\begin{oss}
	For $k_u=k_v=0$, the eigenvalue $0$ is double. This is because we have two different domains with Neumann boundary conditions and so for both we find the zero eigenvalue.
\end{oss}

In Example~\ref{es1}, we refer to Table~\ref{tabeigenu} and to the fact that the first non-zero eigenvalue for $k_v=+\infty=k_u$ is smaller than the one for $k_v=0=k_u$. So, we look for an unstable range such that $\eta_1^\infty \in (\eta_-, \eta_+)$ but $\eta_1^0 \notin (\eta_-, \eta_+)$. Then, we expect to see a different behaviour of the solutions. We perform also an intermediate case in which $k_v$ is small but positive in order to see the evolution in shapes passing from a situation in which there are no unstable modes to another one in which there is only one of them. In Example~\ref{es2}, we show the appearance and the evolution of patterns in both the limiting cases and an intermediate one.


\begin{es}\label{es1}
	We look for some appropriate values of the diffusion coefficients in order to have $\eta_+ \in [D_{vr}\pi^2, D_{vr} 4\pi^2)$. In that way, we expect to see patterns for $k_v\in(0,+\infty]$, since the first eigenvalue is in the unstable range (see Table~\ref{tabeigenu}). Instead, for $k_v=0$, there is any non-zero eigenvalue in $(\eta_-,\eta_+)$, then solutions should converge to the steady state in~\eqref{equileh}.
	Therefore, choosing $\theta=10^{-2}$ in~\eqref{param},
	we infer that $\eta_+=30.01\in [\pi^2, 4\pi^2)$. The results are the following.
	\begin{description}
    \item[Case 1.] We take $k_v=0$ and the other data according to \eqref{param} ($\theta=10^{-2}$, $k_u=0$). For construction, we gain the absence of patterns. 
	\captionsetup[figure]{labelfont=bf,textfont={it}}
	\begin{figure}[H] 	
		\begin{center}
			\includegraphics[scale=0.35]{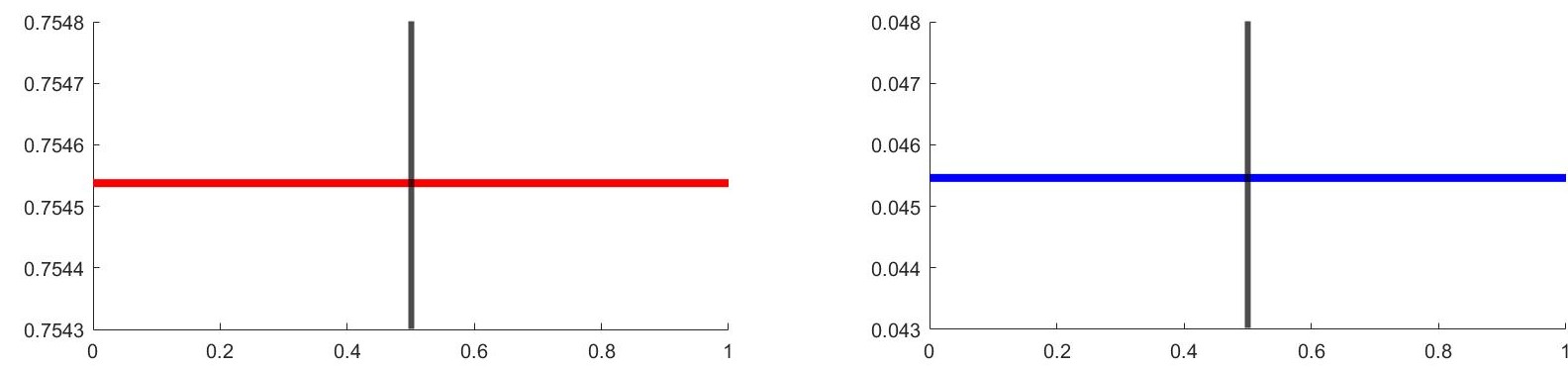}
			\caption{As expected, taking $k_v=0$, we can appreciate the convergence to the steady state $(\overline{u}, \overline{v})$ previously found. Indeed, we choose the data in order to not include positive eigenvalues in the unstable interval $(\eta_-,\eta_+)$ in the case of zero permeability.}
			\label{1k0}
		\end{center}
	\end{figure}
 
    \item[Case 2. ] We take $k_v= 10^{-2}$ and the other data according to \eqref{param} ($\theta=10^{-2}$, $k_u=10^{-6}$). We gain a single unstable mode which is $\eta_1= 0.04$.
    \captionsetup[figure]{labelfont=bf,textfont={it}}
    \begin{figure}[H] 	
    	\begin{center}
    		\includegraphics[scale=0.3]{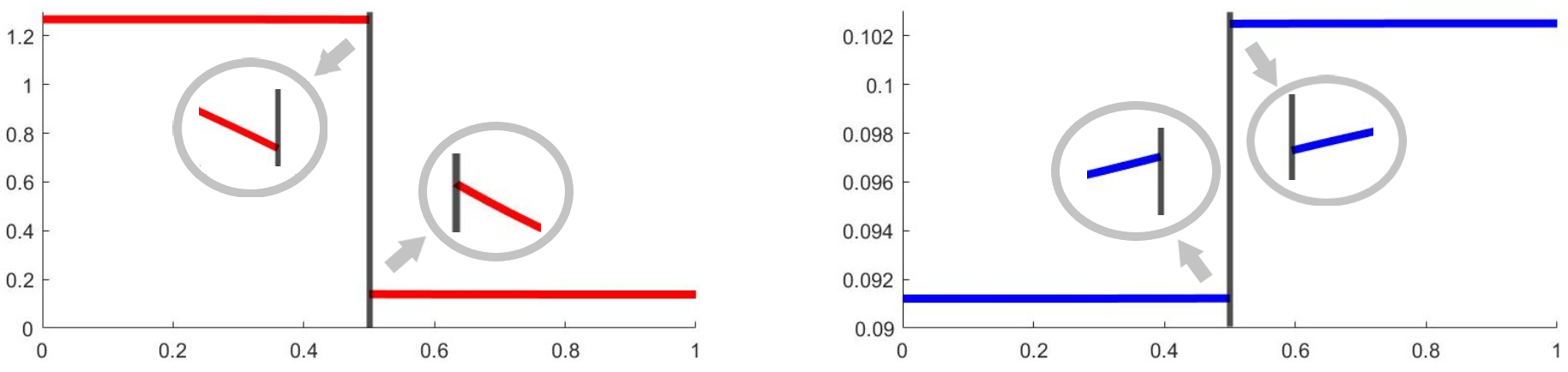}
    		\caption{In the case $k_v=10^{-2}$, we can find a small positive eigenvalue in a neighbourhood of zero which is then in the unstable range $(0,30.01)$. Then, we observe the appearance of a simple pattern which is only a piecewise function with a jump at the membrane. In the zoom circles, we focus the attention on solutions derivatives at the membrane to better appreciate that membrane conditions are satisfied. Moreover, the sign of the derivatives corresponds to the sign of the jump.  }
    		\label{2k1}
    	\end{center}
    \end{figure}  
        
	\item[Case 3.] We consider $k_v= 10^8$ and the other data according to \eqref{param} ($\theta=10^{-2}$, $k_u = 10^{4}$).
	\captionsetup[figure]{labelfont=bf,textfont={it}}
	\begin{figure}[H] 	
		\begin{center}
			\includegraphics[scale=0.35]{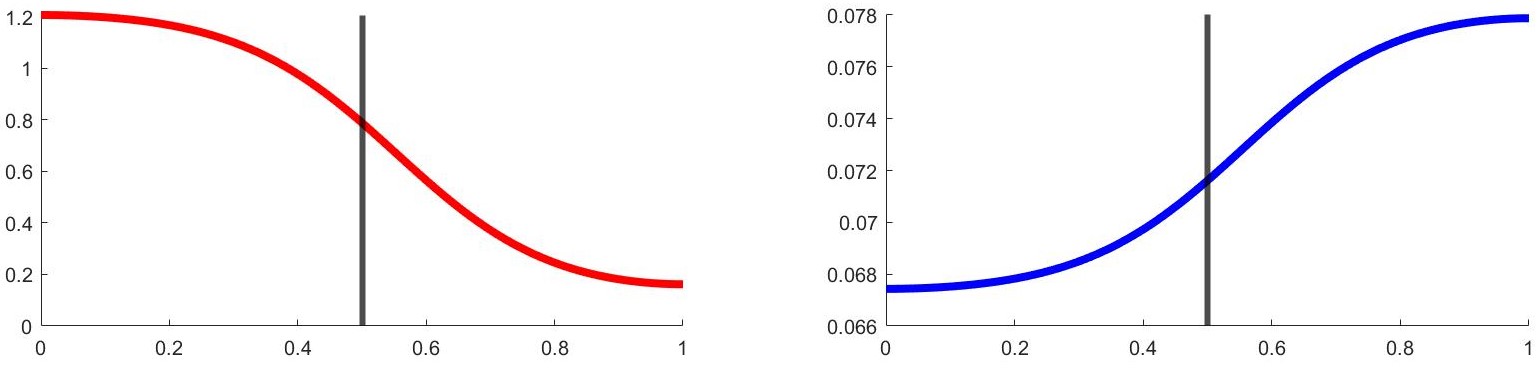}
			\caption{As built, for $k_v\!=\!+\infty$, we see the appearance of continuous patterns, since the permeability coefficients are really big. Indeed, the shape corresponds to the one seen in Figure~\ref{2k1} but, at the membrane, the jump is now reduced to zero.}
			\label{3kinf}
		\end{center}
	\end{figure}
\end{description}
\end{es}

\begin{es}\label{es2}
	We show the evolution of patterns varying $k_v\in [0,+\infty]$ and fixing $\theta$. We choose the setting of Case $3$ in Figure~\ref{3teta}. Then, we take $\theta=3\cdot 10^{-4}$ in~\eqref{param}.
    	
	\begin{description}
		\item[Case 1. ] We consider $k_v\!=\!0$ and the other parameters according to the data in \eqref{param} ($\theta\!=\!3\cdot\!10^{-4}$, $k_u= 0$). The number of eigenvalues in the unstable interval $(\eta_-,\eta_+)$ is $5$.
	\captionsetup[figure]{labelfont=bf,textfont={it}}
	\begin{figure}[H] 	
		\begin{center}
			\includegraphics[scale=0.35]{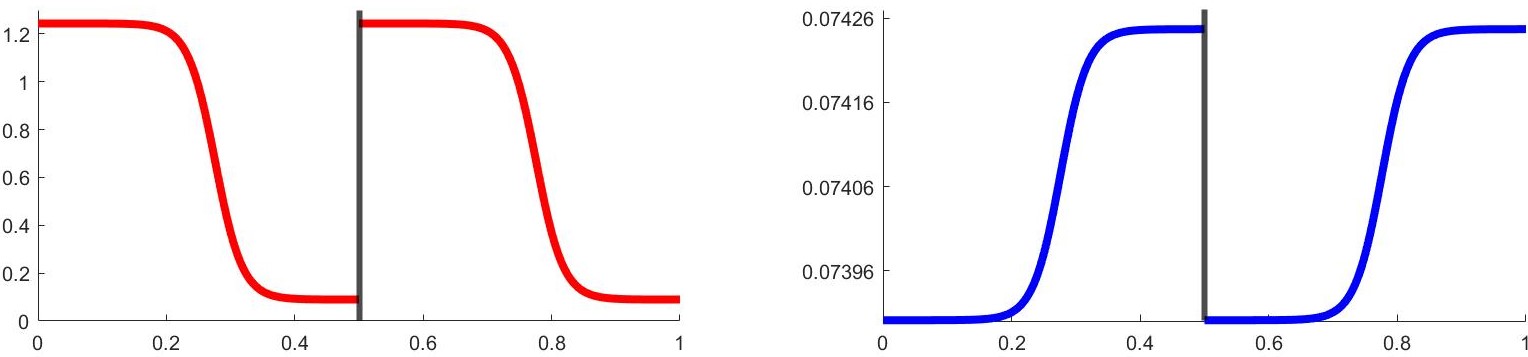}
			\caption{Choosing $k_v=0$, we clearly see patterns for $u$ and $v$. In particular, they are similar to the one observed in Figure~\ref{3teta}. A remarkable difference is at the membrane where Kedem-Katchalsky conditions are broken and they become standard homogeneous Neumann boundary conditions. }
			\label{1k02}
		\end{center}
	\end{figure}
	\item[Case 2.] We take $k_v= 10$ and the other parameters according to \eqref{param} ($\theta=3\cdot 10^{-4}$, $k_u= 3\cdot 10^{-3}$).
	\captionsetup[figure]{labelfont=bf,textfont={it}}
	\begin{figure}[H] 	
		\begin{center}
			\includegraphics[scale=0.35]{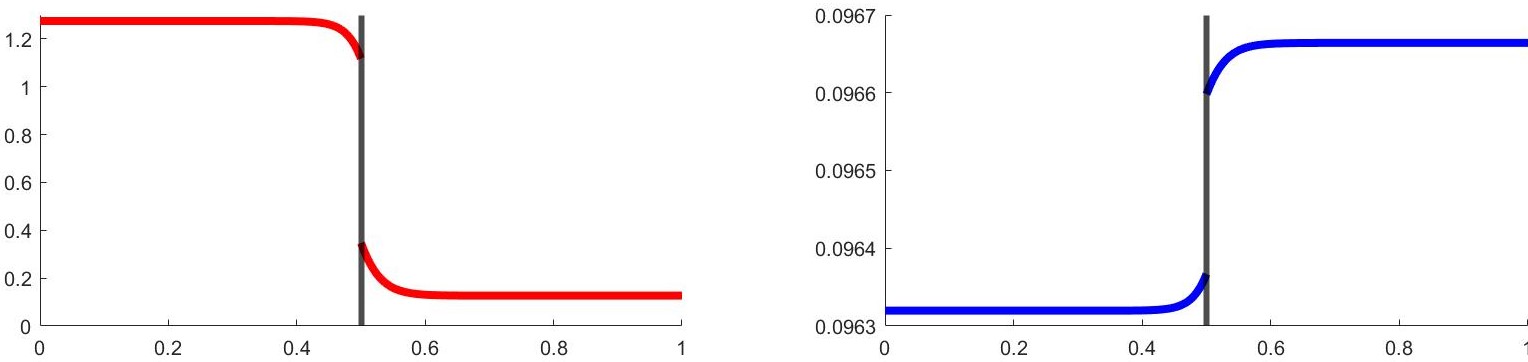}
			\caption{With $k_v=10$, solutions converge to an unexpected shape. There are $6$ unstable modes which are not enough to generate a convergence to a more complex pattern, as it could happen with only $3$ eigenvalues more in the case $\theta=10^{-4}$ (as represented in the summary Table~\ref{tabsumm} in Section~\ref{conclu}).}
			\label{2k10}
		\end{center}
	\end{figure}
	
	\item[Case 3. ] We choose $k_v= 10^{8}$ with the other data as in \eqref{param} ($\theta=3\cdot 10^{-4}$, $k_u=10^{4}$).
	\captionsetup[figure]{labelfont=bf,textfont={it}}
	\begin{figure}[H] 	
		\begin{center}
			\includegraphics[scale=0.35]{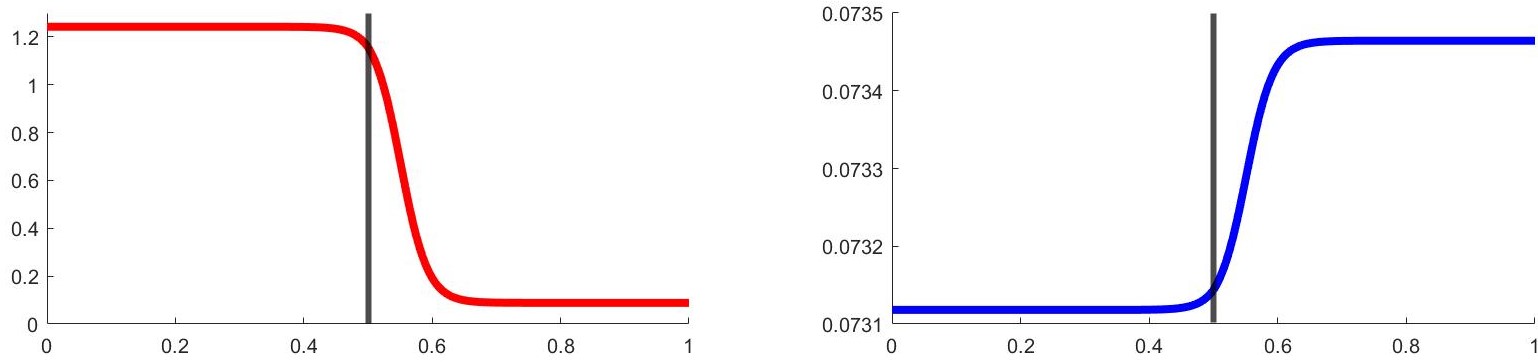}
			\caption{With $k_v$ and $k_u$ sufficiently large, the jump at the membrane (seen in Figure~\ref{2k10}) is reduced to an infinitesimal. Since, the number of unstable modes is small, the same behaviour in Figure~\ref{3kinf} is recover.}
			\label{3kinf2}
		\end{center}
	\end{figure}
\end{description}

\end{es}

To sum up, in this two examples we can observe a particular pattern behaviour, for intermediate $k_v\in(0,+\infty)$ and for a small number of unstable modes, or equivalently, $\theta$ nearby $\theta_c$, which does not occur with smooth Turing instability. Indeed, the transition from the case of two separate domain for $k_v=0$ to a unique entire one for $k_v=+\infty$ is realized through a discontinuous state, which is a nearly constant function with a jump at the membrane.


\subsection{Effect of the parameter {\texorpdfstring{$\eps$}{}} }\label{effeps}
Another interesting parameter is $\eps$, as briefly explained choosing reaction terms in Subsection~\ref{reaction}.
We remember that the smaller we take $\eps$, the faster are the reactions and the more numerous are the patterns. However, in the limit $\eps\rightarrow 0$, Turing instability for fast reaction-diffusion systems turns out to be equivalent to the instability due to backward parabolicity for the limiting cross-diffusion equations, Moussa \textit{et al.}~\cite{moussa}, Perthame and Skrzeczkowski~\cite{jakubperth}. Here, we show the changing of patterns for the solutions $u$ (left) and $v$ (right) decreasing the value of $\eps$ in different membrane scenarios. Again, we consider the data setting presented in Subsection~\ref{reaction}. In particular, we choose data in~\eqref{param} with $\theta= 10^{-4}$ and a varying $\eps$.

As previously stressed, we need to look at the instability interval $(\eta_-, \eta_+)$ in \eqref{rangefg} which increases in size as $\eps$ decreases to zero. This implies that the number of eigenvalues (given by Equation~\eqref{lambdau}) in that interval increases as $\eps$ goes to zero. Then, fixing the membrane permeability $k_v$, we expect to see more complicated shapes as $\eps\rightarrow 0$. Instead, fixing $\eps$ and varying $k_v$, we gain or lose (depending on the $\eps$ value) at most one unstable mode. This is why fixing $\eps$ patterns with different $k_u, k_v$ are comparable.
\begin{description}
	\item[Case 1. ] We consider $k_v=0$ and the other parameters according to data in \eqref{param} ($\theta= 10^{-4}$, $k_u=0$, $\eps$ varies). Indeed, we have not communicating domains in which we consider a reaction-diffusion system with reaction that is faster decreasing $\eps$.
\newpage	
	$\underline{\eps =10.}$
	
	\captionsetup[figure]{labelfont=bf,textfont={it}}
	\begin{figure}[H] 	
		\begin{center}
			\includegraphics[scale=0.35]{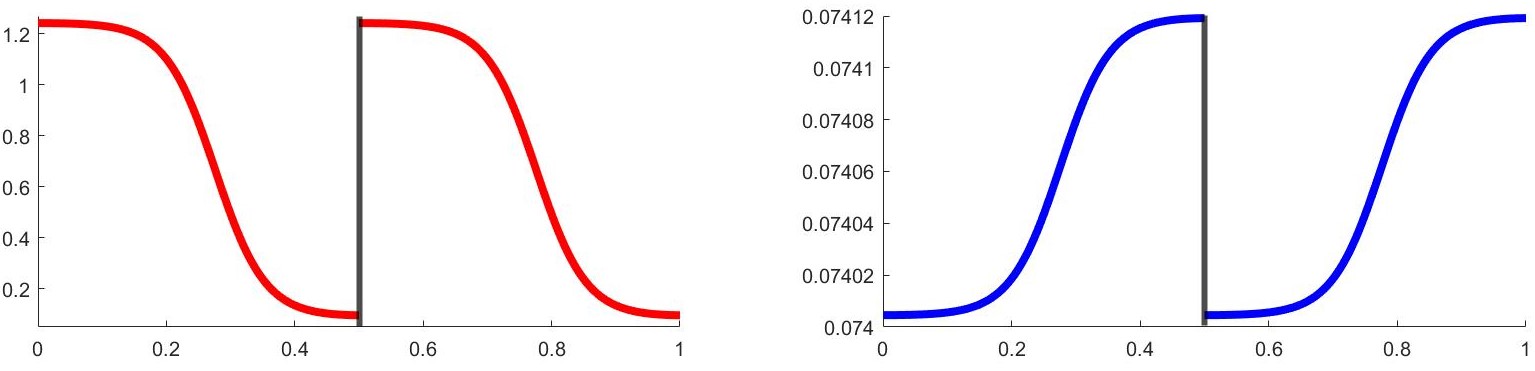}
			\caption{We represent the convergent solutions for $\eps=10$. Diffusion prevails over reaction, then solutions are smooth and we can appreciate the emergence of patterns. }
			\label{0eps10}	
		\end{center}
	\end{figure}
	
	$\underline{\eps =1.}$
	
	\captionsetup[figure]{labelfont=bf,textfont={it}}
	\begin{figure}[H] 	
		\begin{center}
			\includegraphics[scale=0.35]{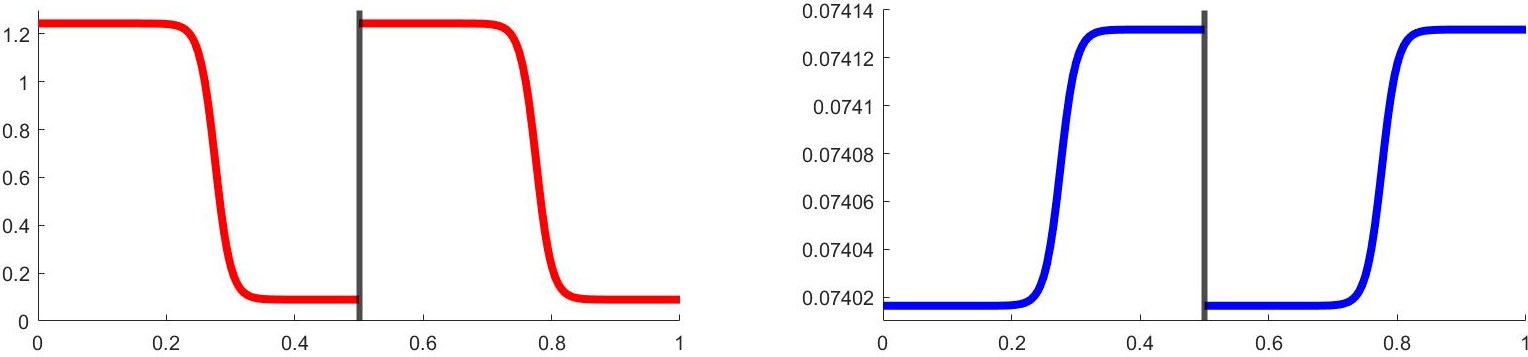}
			\caption{In the case $\eps=1$, solutions does not change significantly respect to $\eps=10$ (we have only $5$ unstable modes) but the slope is increasing. This scenario corresponds to the standard reaction-diffusion diffusion one analysed until now.}
			\label{0eps1}	
		\end{center}
	\end{figure}
	
	$\underline{\eps =1/5.}$
	
	\captionsetup[figure]{labelfont=bf,textfont={it}}
	\begin{figure}[H] 	
		\begin{center}
			\includegraphics[scale=0.35]{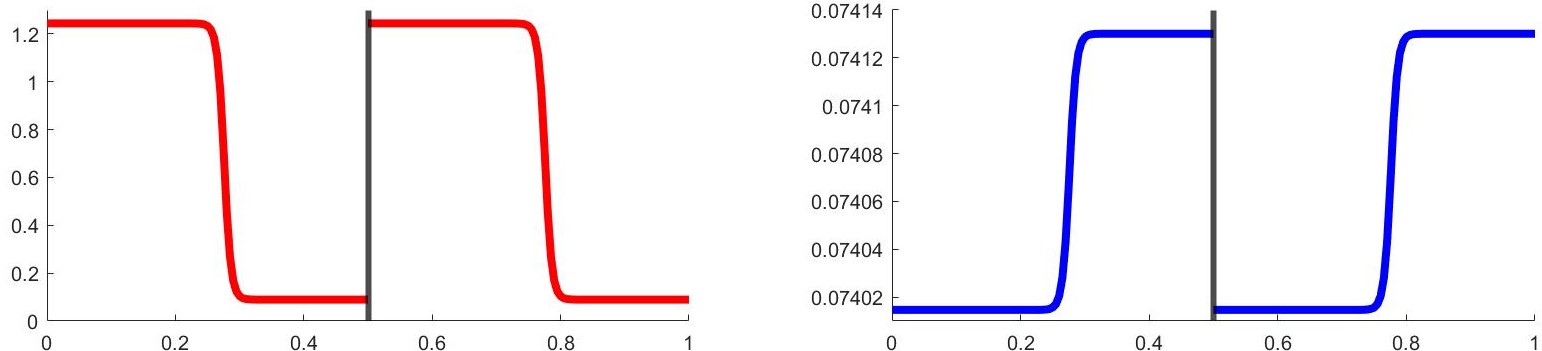}
			\caption{It is with $\eps= 1/5$ that we can see that the patterns are becoming more discontinuous, since numerically we are approaching the zero limit.}
			\label{0eps5}	
		\end{center}
	\end{figure}
\newpage	
	$\underline{\eps =1/20.}$
	
	\captionsetup[figure]{labelfont=bf,textfont={it}}
	\begin{figure}[H] 	
		\begin{center}
			\includegraphics[scale=0.35]{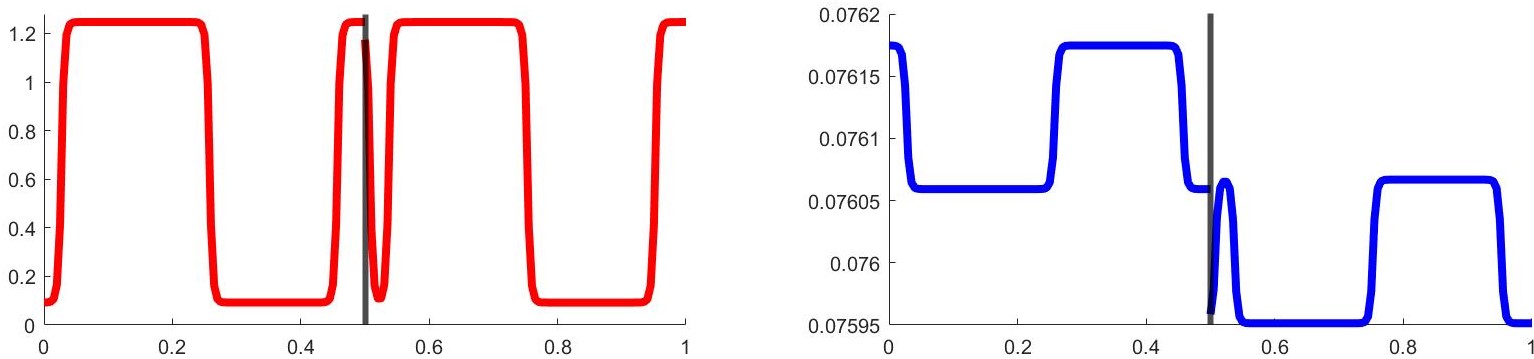}
			\caption{With $\eps= 1/20$, high frequency of oscillations are clearly appreciated. Numerically, we are converging to zero and then Turing instability is equivalent to instability and discontinuity of the ill-posedness of the backward parabolicity for the cross-diffusion system.}
			\label{0eps20}	
		\end{center}
	\end{figure}
	
	$\underline{\eps =1/100.}$
	
	\captionsetup[figure]{labelfont=bf,textfont={it}}
	\begin{figure}[H] 	
		\begin{center}
			\includegraphics[scale=0.35]{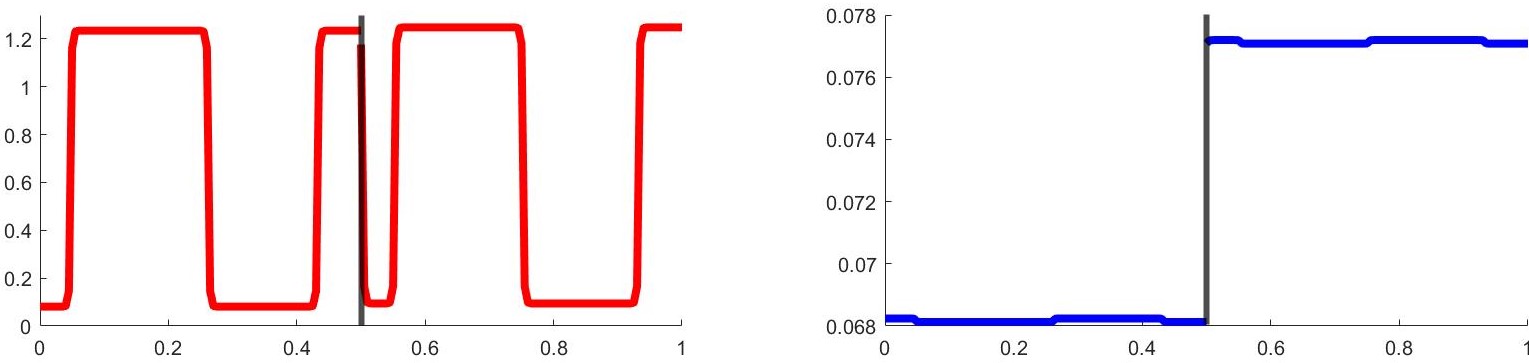}
			\caption{Discontinuities are dominant with $\eps=1/100$. The right picture representing $v$ has similar shapes has the one for $\eps=1/20$ but here the jump is more remarkable. The number of eigenvalues in the unstable range is really high and the slope in the patterns is diverging. We are far away from the smooth and regular patterns observed with slower reactions.}
			\label{0eps100}	
		\end{center}
	\end{figure}	
	
	\item[Case 2. ] We consider $k_v=1$ and the other parameters according to data in \eqref{param} ($\theta= 10^{-4}$, $k_u=10^{-4}$, $\eps$ varies). The passage through the membrane is now allowed. 
\newpage	
	$\underline{\eps =10.}$
	
	\captionsetup[figure]{labelfont=bf,textfont={it}}
	\begin{figure}[H] 	
		\begin{center}
			\includegraphics[scale=0.35]{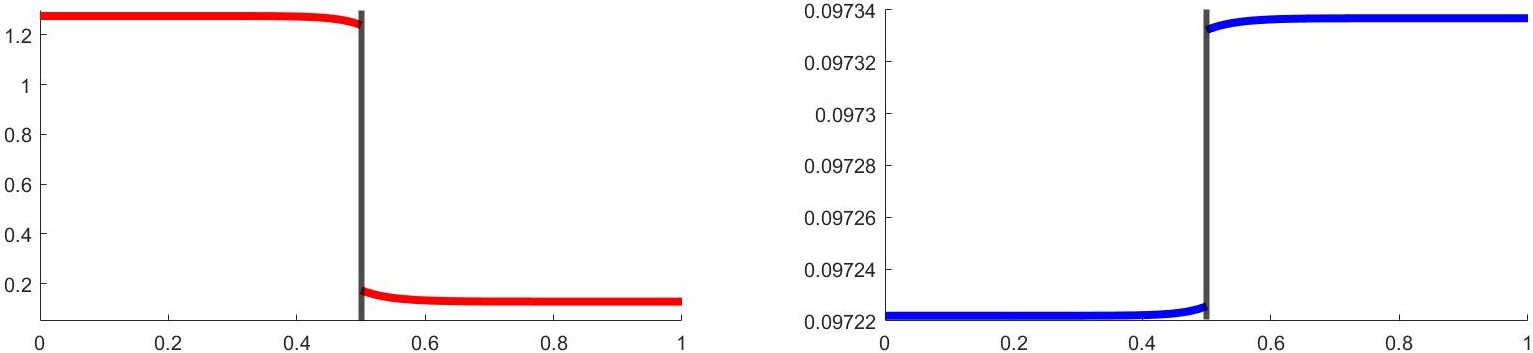}
			\caption{With $\eps=10$, the slow reaction is not prevailing significantly on the diffusion (since increasing the value of $\eps$, reactions converge to zero). The permeability of the membrane promotes dissipation but a slope nearby the interface is still observed.}
			\label{1eps10}	
		\end{center}
	\end{figure}
	
	$\underline{\eps =1.}$
	
	\captionsetup[figure]{labelfont=bf,textfont={it}}
	\begin{figure}[H] 	
		\begin{center}
			\includegraphics[scale=0.35]{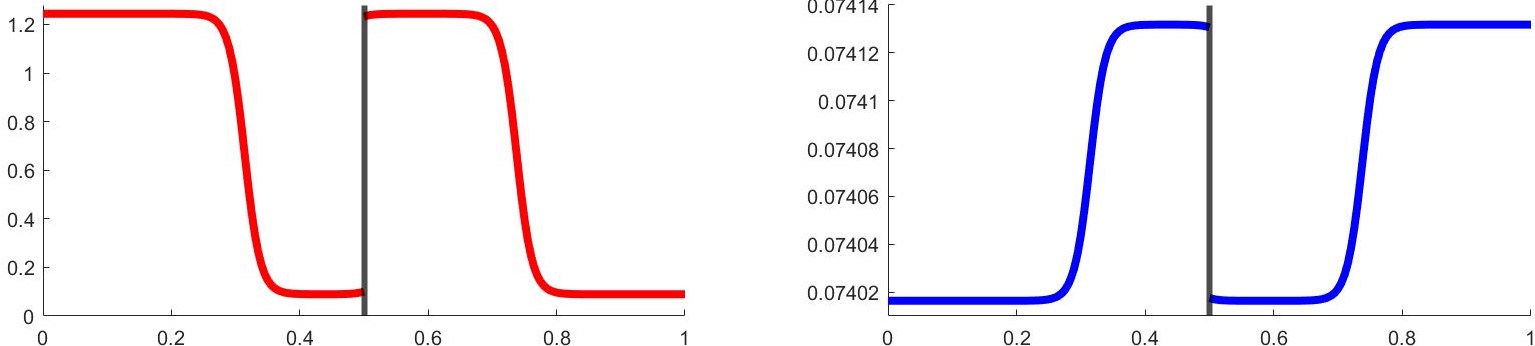}
			\caption{Coming back to a standard reaction-diffusion equation with $\eps=1$, we observe a similar shape as in the case $k_v=0$ but we can appreciate a little slope nearby the membrane.}
			\label{1eps1}	
		\end{center}
	\end{figure}
	
	$\underline{\eps =1/5.}$
	
	\captionsetup[figure]{labelfont=bf,textfont={it}}
	\begin{figure}[H] 	
		\begin{center}
			\includegraphics[scale=0.35]{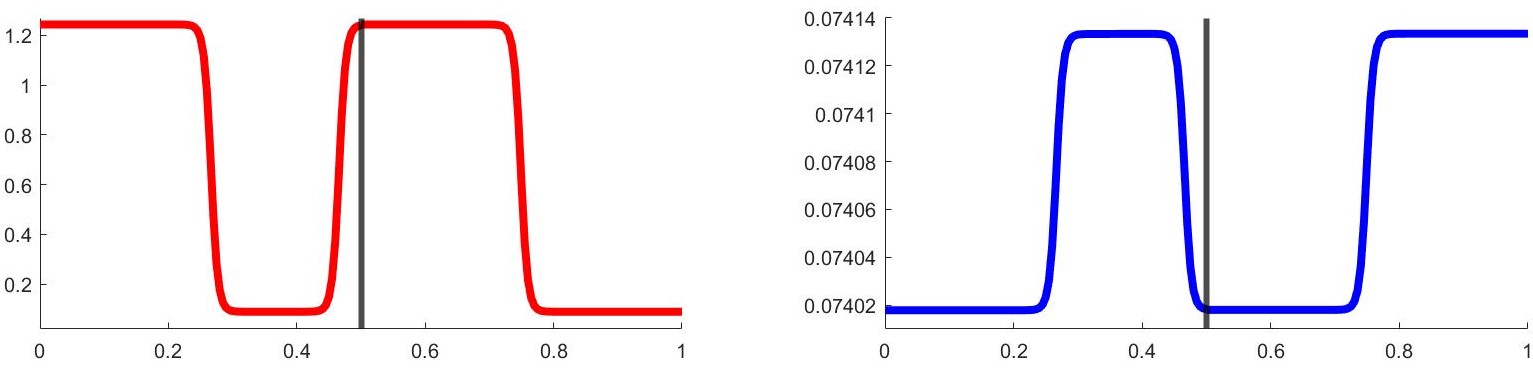}
			\caption{Reducing $\eps$, slopes increase but the jump at the membrane is less significant since membrane derivatives are really small with the data chosen. }
			\label{1eps5}	
		\end{center}
	\end{figure}
\newpage	
	$\underline{\eps =1/20.}$
	
	\captionsetup[figure]{labelfont=bf,textfont={it}}
	\begin{figure}[H] 	
		\begin{center}
			\includegraphics[scale=0.35]{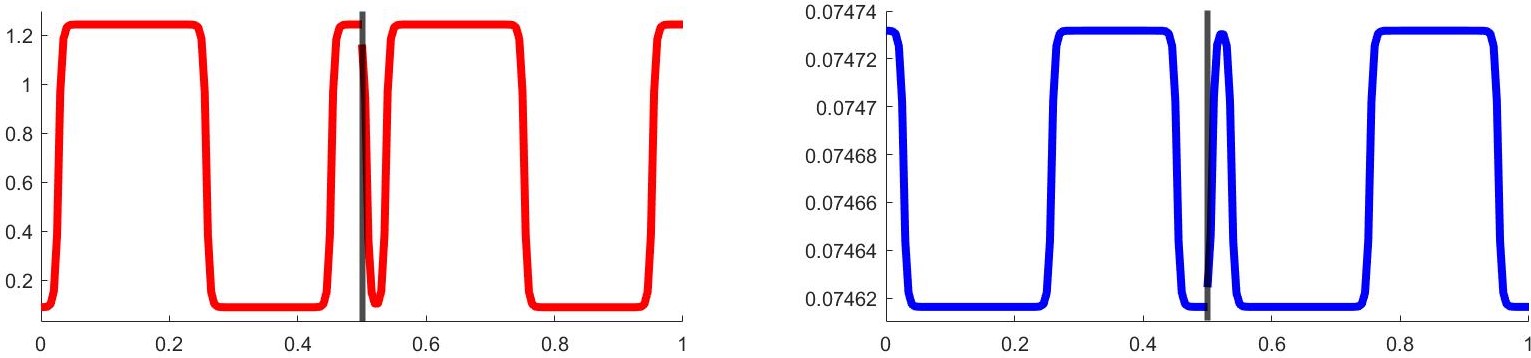}
			\caption{As in the case $k_v=0$, oscillations are increasing respect to Figure~\ref{1eps5}.}
			\label{1eps20}	
		\end{center}
	\end{figure}
	
	$\underline{\eps =1/100.}$
	
	\captionsetup[figure]{labelfont=bf,textfont={it}}
	\begin{figure}[H] 	
		\begin{center}
			\includegraphics[scale=0.35]{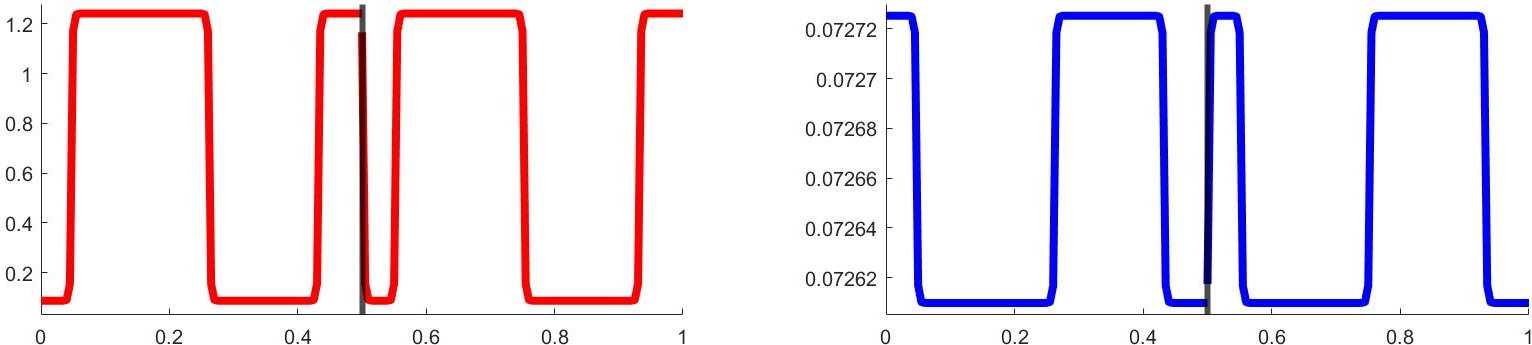}
			\caption{Taking $\eps=1/100$ and $k_v>0$, instabilities are dominant and patterns for $v$ (in the right) are more remarkable than in the case $k_v=0$, even if the shape is still unchanged.}
			\label{1eps100}	
		\end{center}
	\end{figure}
	
	\item[Case 3. ] We consider $k_v=10^8$ and the other parameters according to data in \eqref{param} ($\theta= 10^{-4}$, $k_u=10^{4}$, $\eps$ varies). We remember that the membrane is fully permeable and then we observe a reaction-diffusion system on the whole interval $[0,1]$, since membrane conditions are reduced to continuity conditions.
	
	$\underline{\eps =10.}$
	
	\captionsetup[figure]{labelfont=bf,textfont={it}}
	\begin{figure}[H] 	
		\begin{center}
			\includegraphics[scale=0.35]{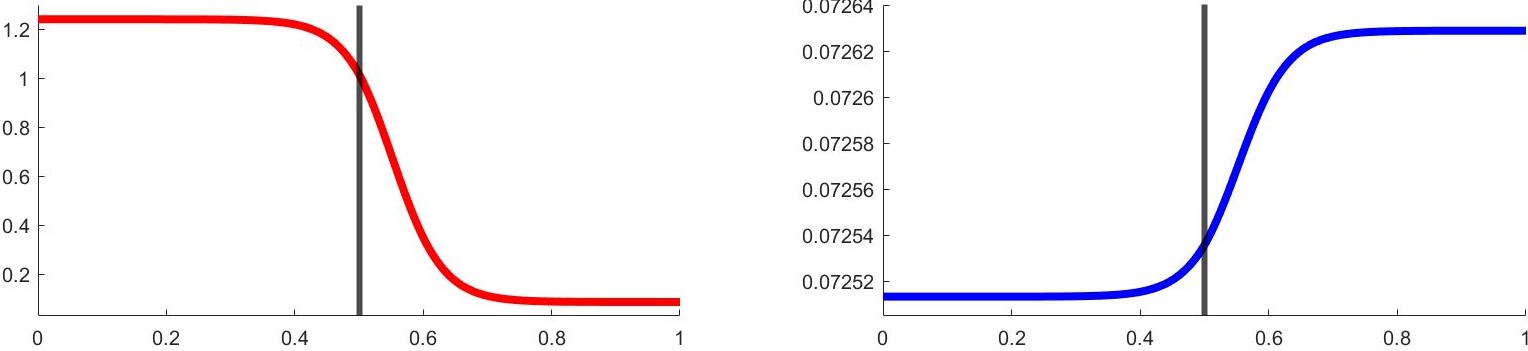}
			\caption{The jump between the right and left side solutions in Figure~\ref{1eps10} is now filled and we can observe continuous solutions. }
			\label{infeps10}	
		\end{center}
	\end{figure}
\newpage	
	$\underline{\eps =1.}$
	
	\captionsetup[figure]{labelfont=bf,textfont={it}}
	\begin{figure}[H] 	
		\begin{center}
			\includegraphics[scale=0.35]{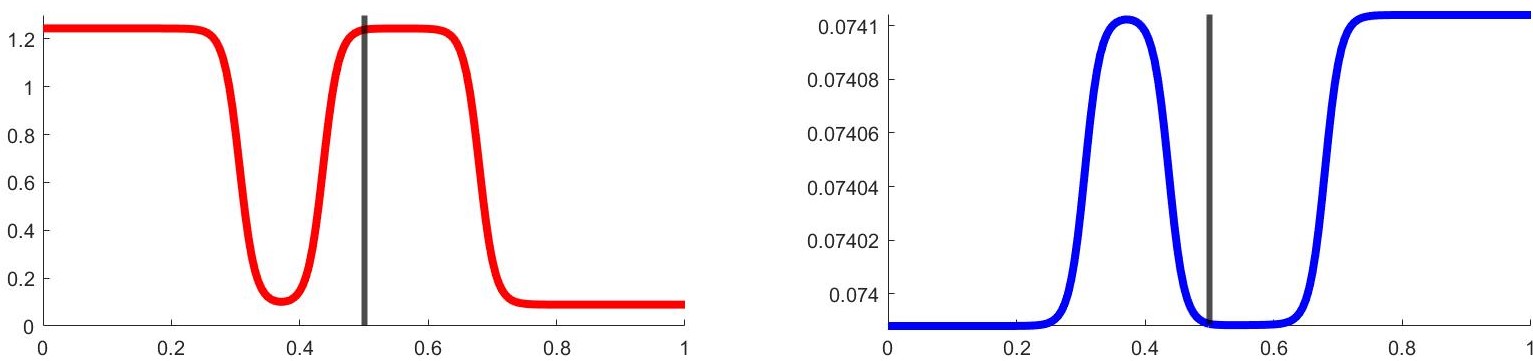}
			\caption{With $\eps=1$, the continuous solutions are similar to the following case $\eps=1/5$ but they are more regular.}
			\label{infeps1}	
		\end{center}
	\end{figure}
	
	$\underline{\eps =1/5.}$
	
	\captionsetup[figure]{labelfont=bf,textfont={it}}
	\begin{figure}[H] 	
		\begin{center}
			\includegraphics[scale=0.35]{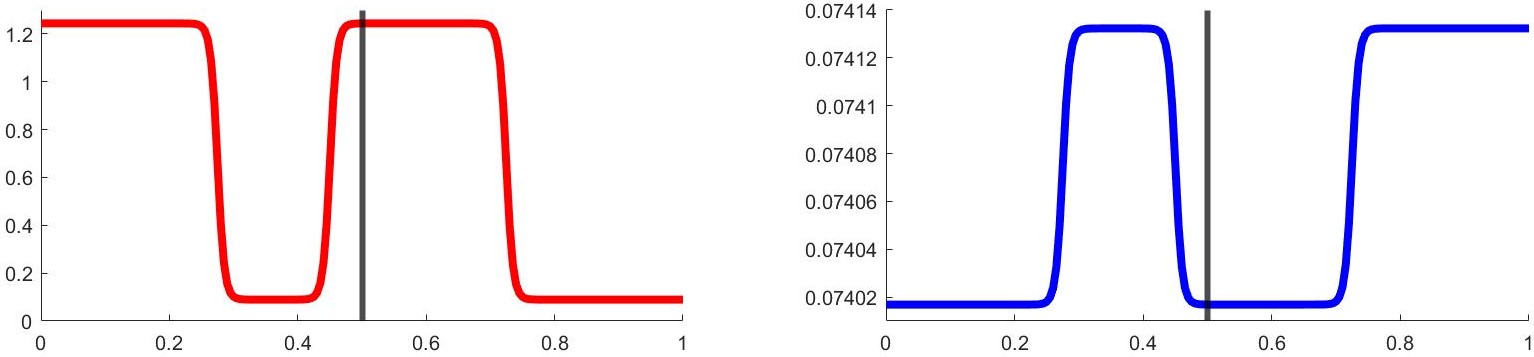}
			\caption{With $\eps=1/5$, pictures can be well predicted from Figure~\ref{1eps5}.  }
			\label{infeps5}	
		\end{center}
	\end{figure}
	
	$\underline{\eps =1/20.}$
	
	\captionsetup[figure]{labelfont=bf,textfont={it}}
	\begin{figure}[H] 	
		\begin{center}
			\includegraphics[scale=0.35]{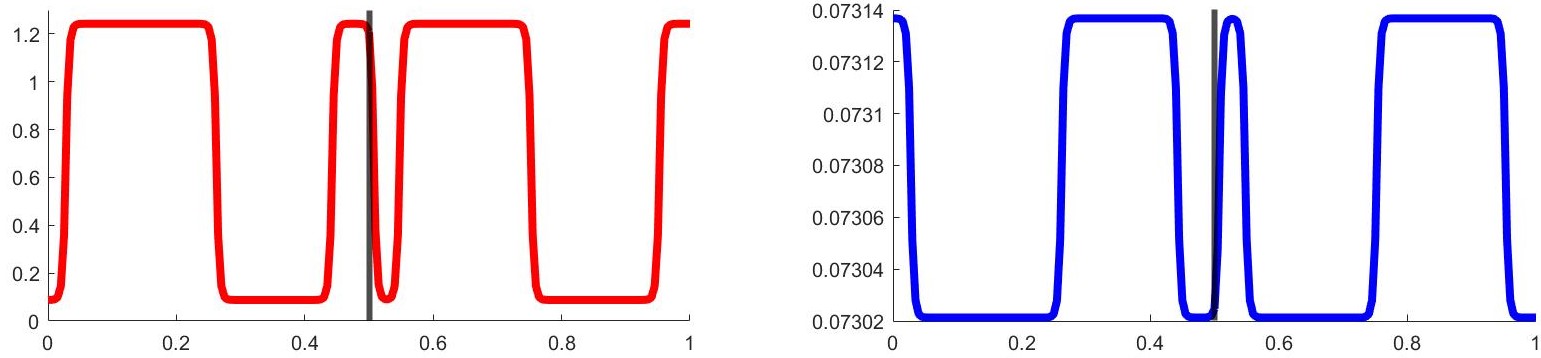}
			\caption{Again with $\eps= 1/20$, we are approaching the zero numerical limit. Then, the appearance of membrane continuous, but not smooth instabilities can be observed in both $u$ and~$v$. }
			\label{infeps20}	
		\end{center}
	\end{figure}
\newpage	
	$\underline{\eps =1/100.}$
	
	\captionsetup[figure]{labelfont=bf,textfont={it}}
	\begin{figure}[H] 	
		\begin{center}
			\includegraphics[scale=0.35]{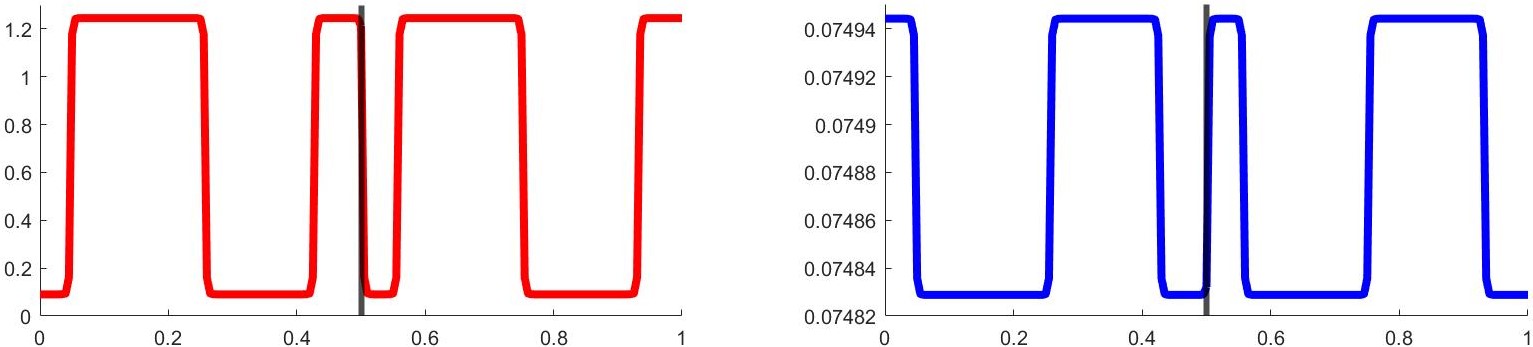}
			\caption{For $\eps=1/100$, oscillations are now continuous at the membrane respect to Figure~\ref{1eps100}}
			\label{infeps100}	
		\end{center}
	\end{figure}
	
\end{description}

Finally, as $\eps$ converges to zero, we numerically observe convergence to instability due to backward parabolicity for the limiting cross-diffusion equations. Indeed, we remark that, from a numerical point of view, the convergence to zero is already attained with $\eps=1/100$. Fixing $\eps$ and varying $k_v$, we observe similar behaviour as in the previous subsections.  

\section{Conclusions}\label{conclu}

Turing instability for a standard reaction-diffusion problem is known to be a universal mechanism for pattern formation. We questioned the effect on pattern formation of a permeable membrane at which we have dissipative conditions. This interest follows both a path started in the study of membrane problems, Ciavolella \textit{et al.}~\cite{ciadapou}, Ciavolella and Perthame~\cite{ciavper}, and their importance in biology.
Then, we have studied Turing instability from both an analytical and a numerical point of view for a reaction-diffusion membrane problem of two species $u$ and $v$ as in~\eqref{eq}.

Our method relies on a diagonalization theory for membrane operators. A detailed proof of related results in Appendix~\ref{proofdiag} is left to more analytical studies.
Thanks to this theory, in Section~\ref{turing}, we could perform an analogous analysis of Turing instability as in the standard case without membrane under the hypothesis to have equal eigenfunctions for the membrane Laplace operator associated to the two species. This condition is related, thanks to Lemma~\ref{lemmacond}, to restrictions  \eqref{nukteta} and \eqref{etalambda}.
We left as an open problem the identification of cases in which these constraints can be eliminated. 

In order to pass to the numerical analysis, we have introduced in Section~\ref{1d} the one dimensional problem and the explicit solutions of the eigenvalue problem. 
Membrane Laplace eigenvalues are implicitly defined by Equation~\eqref{lambdau}, since we have chosen to introduce the condition $\nu_D=1$. This could be avoided under biological reasons considering, then, Equation~\eqref{lambda}. Moreover, choosing a proper domain, it is possible to extend the analyses in the two-dimensional case.

Concerning numerical examples in Section~\ref{examples}, it is possible to take more complex and more realistic data. A more extensive study, with other nonlinearities, is of interest. Moreover, we have fixed the diffusion coefficient $D_v$ whose role is of interest also. 

In Table~\ref{tabsumm}, we sum up the different patterns observed in Subsection~\ref{teta}~and~\ref{diffk}, decreasing the diffusion ratio $\theta\!~=~\!\frac{D_{ul}}{D_{ur}}\!~=~\!\frac{D_{vl}}{D_{vr}}$ from the critical value $\theta_c= 3.1\cdot 10^{-1}$ (from left to right in the rows) and increasing the permeability coefficient values $k_v\in[0,+\infty]$ (from top to down in the columns). We consider only the activator $u$ and we take reaction terms as in \eqref{react}, initial data as in Figure~\ref{initdata} and data setting as in \eqref{param}. We recall that the spatial interval of study is $[0,1]$.
We stress on the fact that the first ($k_v=0$) and last ($k_v=+\infty$) row correspond to Turing instabilities observed in a reaction-diffusion problem on a half domain and on the full one respectively. Hence, it is coherent that decreasing $\theta$ the number of patterns increases in the biggest domain. 

\captionsetup[table]{labelfont=bf,textfont={it}} 
\begin{table}[H]
	\centering
	\begin{tabular}{c c c c c c}
		\toprule
		\diagbox{$k_v$}{$\theta$}  & $\theta_c$                     & $10^{-2}$                         & $10^{-3}$                         & $10^{-4}$                           & $10^{-5}$\\
		\cmidrule(r){1-1}\cmidrule(lr){2-2}\cmidrule(lr){3-3}\cmidrule(lr){4-4}\cmidrule(lr){5-5}\cmidrule(l){6-6}
		$0$             &                                & \includegraphics[scale=0.118]{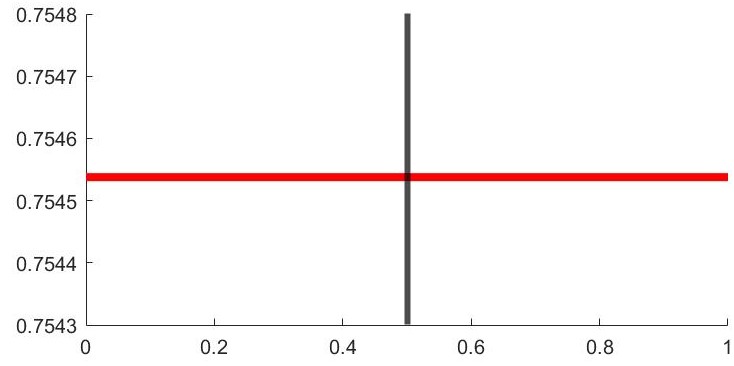} & \includegraphics[scale=0.118]{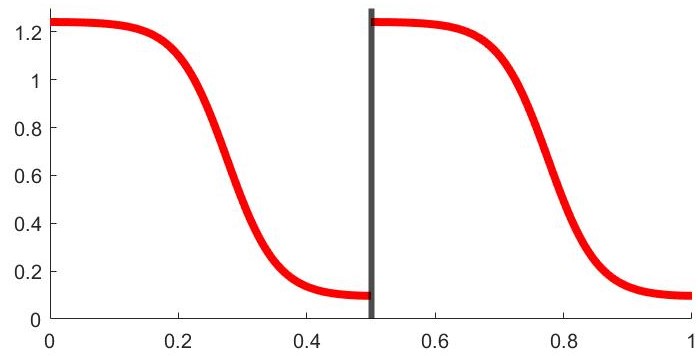} & \includegraphics[scale=0.118]{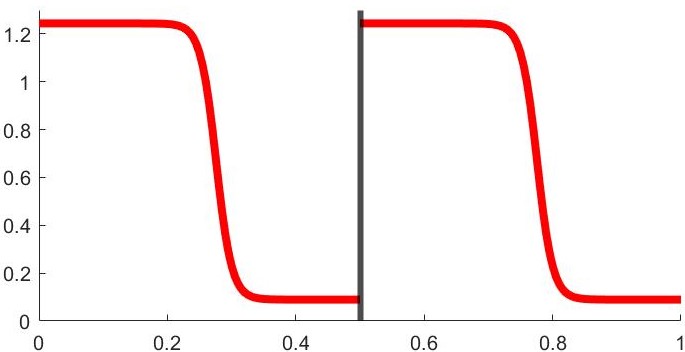}  &\includegraphics[scale=0.118]{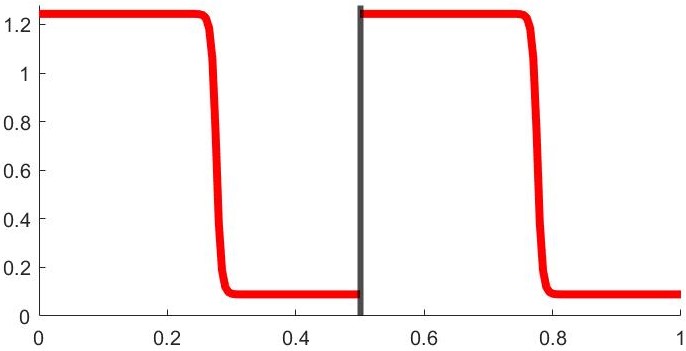} \\[1ex]
		$1 $            & \includegraphics[scale=0.118]{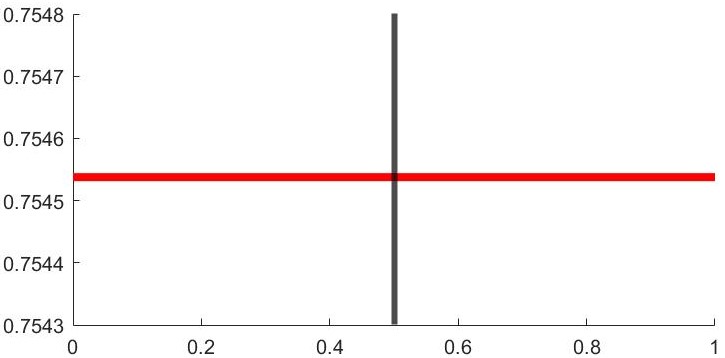} & \includegraphics[scale=0.118]{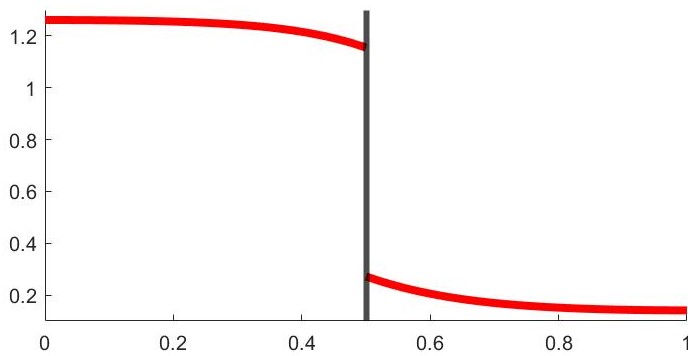} & \includegraphics[scale=0.118]{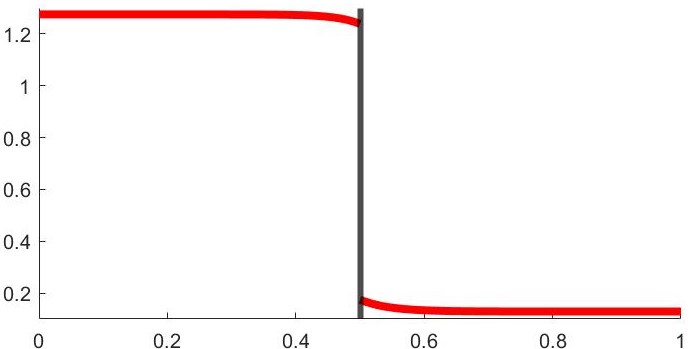} & \includegraphics[scale=0.118]{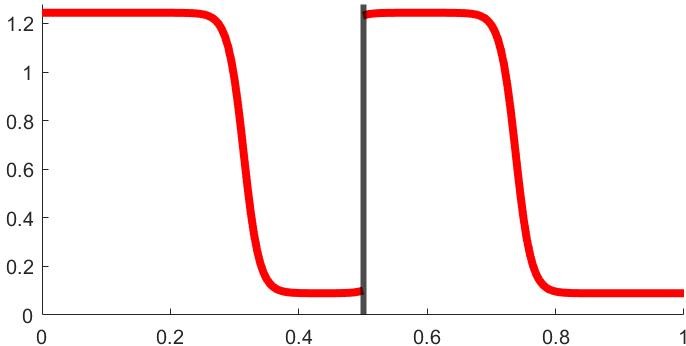}  &\includegraphics[scale=0.118]{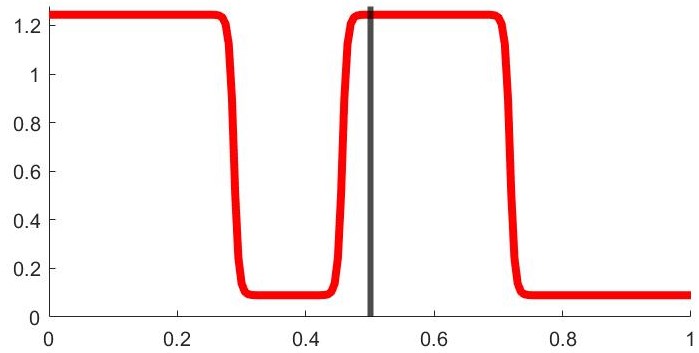}\\[1ex]
		$+\infty$       &                                & \includegraphics[scale=0.118]{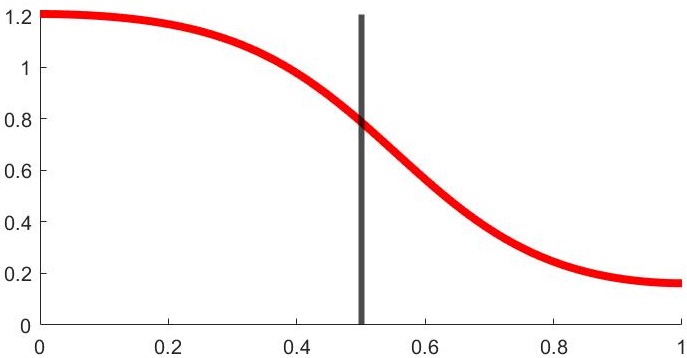} & \includegraphics[scale=0.118]{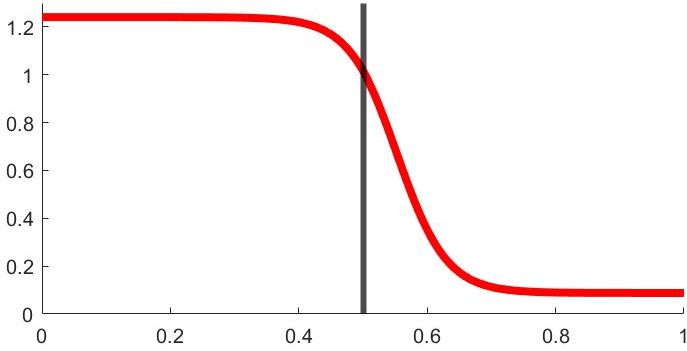} &\includegraphics[scale=0.118]{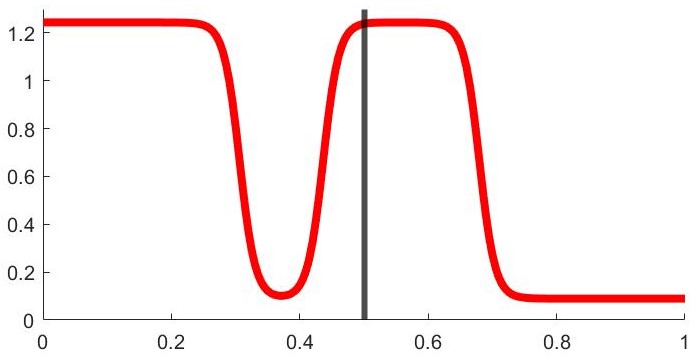}  &\includegraphics[scale=0.118]{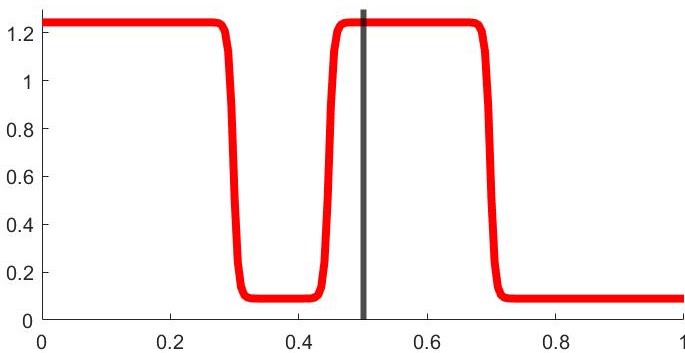}\\[1ex]
		\hline
	\end{tabular}
	\caption{We summarise the evolution of patterns varying $\theta$ and $k_v$.The first column corresponds to the value $\boldsymbol{\theta=\theta_c}$, in which case there are no unstable modes. Then, for all $k_v$, convergence to the steady state is observed. 
	For $\boldsymbol{\theta = 10^{-2}}$, we have again 
	zero eigenvalues for $k_v=0$ and one eigenvalue for $k_v\in(0,+\infty]$. So, we observe convergence respectively to a steady state and a simple pattern, discontinuous in the case $k_v=1$. 
    For $\theta$ small enough ($\boldsymbol{\theta=10^{-3}, 10^{-4}, 10^{-5}}$), we observe more complex patterns with the main discontinuity property in the case of a non-trivial $k_v$ (second row). We remark that in the picture for $\theta=10^{-5}$ and $k_v=1$, the jump is really small compared to the axis scale (see Figure~\ref{4teta}).     }
	\label{tabsumm}
\end{table}
Surprisingly, not only adding diffusion but also adding dissipative membrane conditions, we observe the equilibria
stability's break. As in the classical Turing analysis, decreasing $\theta$, we get more complex patterns. Contrary to standard Turing instability, with non-trivial membrane permeability,
discontinuity at the membrane characterizes the steady state. Moreover, for $\theta$ in a neighbourhood of $\theta_c$ and $k_v\in (0,+\infty)$, a singular pattern appears. Indeed, it is a simple nearly constant function with a jump at the membrane.

In Subsection~\ref{effeps}, we have numerically studied a fast reaction-diffusion membrane system, leaving a rigorous analysis as an open problem. Again, discontinuity characterizes instability for $k_v\in(0,+\infty)$.
  
\section*{Acknowledgements}

The author has received funding from the European Research Council (ERC) under the European Union's Horizon 2020 research and innovation programme (grant agreement No 740623). The work was also partially supported by GNAMPA-INdAM.

\appendix

\section{Diagonalization theory on membrane operators}\label{proofdiag}

We introduce the diagonalization result, Brezis~\cite{brezis}, Evans~\cite{evans}, for membrane operators which assures the existence of a sequence of eigenvalues and eigenfunctions that solve each problem in \eqref{eigenw} and \eqref{eigenz}. 
\begin{thm}[Diagonalization theorem for compact, self-adjoint membrane operators.]\label{diagon}
	Let $A$ be a compact, self-adjoint membrane operator on a separable Hilbert space $\mathcal{H}$ with infinite dimension. There exists a sequence of real numbers $\{\lambda_n\}_{n\in \N}$ such that $\{|\lambda_n|\}_{n\in \N}$	is non increasing, converges to zero and such that:
	\begin{itemize}
		\item for any $n$ such that $\lambda_n$ is non-zero, $\lambda_n$ is an eigenvalue of $A$ and $E_n:= \ker(A-\lambda_n I)$ is a subspace of $\mathcal{H}$ with finite dimension; moreover, if $\lambda_n$ and $\lambda_m$ are distinct, their corresponding eigenspaces are orthogonal;
		\item if $E:= \mbox{Span}\hspace{-9pt} \bigcup\limits_{\tiny\begin{array}{cc}
				n\in\N\\
				\lambda_n\neq 0
		\end{array}}\hspace{-9pt} E_n$, then $\ker(A)=E^\perp$;
	\end{itemize} 
\end{thm}
\noindent Indeed, \hypertarget{thmdiag}{Theorem}~\ref{diagon} applies to the inverse operators $L^{-1}$ and $\widetilde{L}^{-1}$. Therefore, we can find also for $L$ and $\widetilde{L}$ a sequence of eigenvalues and a basis of eigenfunctions.

We show here below that the inverse operators verify the hypothesis of this theorem. At first, we introduce the bilinear forms associated to the membrane operators. Then, we prove the hypothesis of the Lax-Milgram Theorem. The following definition is requested.

\begin{defin}\label{defh1}
We define the Hilbert space of functions ${\bf H^1} = H^1(\Omega_l)\times H^1(\Omega_r)$.
We endow it with the norm 
$$\|w\|_{\bf H^1}= \left(\|w^1\|^2_{H^1(\Omega_l)} +\|w^2\|^2_{H^1(\Omega_r)}\right)^\frac{1}{2}.$$
We let $(\cdot,\cdot)_{\bf H^1}$ be the inner product in ${\bf H^1}$.
\end{defin} 

We define the bilinear forms associated with these membrane elliptic operators as 
\begin{equation}
\begin{array}{ll}
B[\varphi,\phi]=\int_{\Omega_l} D_{ul} \nabla \varphi_l \nabla \phi_l+\int_{\Omega_r} D_{ur} \nabla \varphi_r \nabla \phi_r +\int_{\Gamma} D_u k_u (\varphi_r-\varphi_l) (\phi_r-\phi_l),\\[2ex]
\widetilde{B}[\varphi,\phi]=\int_{\Omega_l} D_{vl} \nabla \varphi_l \nabla \phi_l+\int_{\Omega_r} D_{vr} \nabla \varphi_r \nabla \phi_r +\int_{\Gamma} D_v k_v (\varphi_r-\varphi_l) (\phi_r-\phi_l),
\end{array}  
\end{equation}
for $\varphi, \phi \in {\bf H^1}$. We remark that $B$ and $\widetilde{B}$ are symmetric.
For simplicity, we consider the membrane operator $L$. We can follow the same steps for $\widetilde{L}$.
We want to apply the Lax-Milgram theory, Brezis~\cite{brezis}, Evans~\cite{evans}.   
We can readily check continuity and coercivity for $B$. \\
{\it $B$ is continuous.} Thanks to the Cauchy-Schwarz inequality and the continuity of the trace, we can write
\begin{align*}
|B[\varphi,\phi]|&\leq \sum_{\lambda=l,r} ( \; D_{u\lambda}\nl{\nabla \varphi_\lambda}{\Omega_\lambda} \nl{\nabla \phi_\lambda}{\Omega_\lambda} +D_{u\lambda} k_u \nl{[\varphi]}{\Gamma} \nl{[\phi]}{\Gamma}\; )\\ &\leq \sum_{\lambda,\sigma=l,r}\left(\; \nh{\varphi_\lambda}{\Omega_\lambda}\nh{\phi_\lambda}{\Omega_\lambda}+D_{u\lambda} k_u \nh{\varphi_\lambda}{\Omega_\lambda} \nh{\phi_\sigma}{\Omega_\sigma}\; \right)\\
&\leq C\|\varphi\|_{\bf H^1}\|\phi\|_{\bf H^1},
\end{align*}\\	
{\it $B$ is coercive.} Indeed, if we assume \hypertarget{mean}{$\int_{\Omega_l \cup \Omega_r} \varphi =0$}, we can estimate
$$
B[\varphi,\varphi] = \int_{\Omega_l} |\nabla \varphi|^2 + \int_{\Omega_r} |\nabla \varphi|^2+ \int_\Gamma k_i |\varphi_r-\varphi_l|^2 \geq C \|\varphi\|_{\bf H^1}^2
$$
with a membrane version of the Poincaré-Wirtinger inequality on a product space (this theory would not be analysed in this chapter since it is more a functional analysis result which is not of main interest in Turing theory).
With the same assumption, we can check continuity and coercivity of $\widetilde{B}$. 
Therefore, the Lax-Milgram theory applies in this context assuming that $\int_{\Omega_l \cup \Omega_r} w =0$. Then there exists a unique function $w\in {\bf H^1}$ solving 
\begin{equation}\label{lm}
B[w,\varphi]=(\lambda w, \varphi)_{\bf H^1}, \quad \forall \varphi\in {\bf H^1}.
\end{equation}
Whenever \eqref{lm} holds, we write 
\begin{equation*}
w=\lambda L^{-1}w.
\end{equation*}
The inverse operator $L^{-1}: (\mathbf{H^1})^{-1} \rightarrow \mathbf{H^1}$ is a compact operator in $L^2(\Omega_l)\times L^2(\Omega_r)$, since according to the Rellich-Kondrachov theorem $\mathbf{H^1}\subset\subset {\bf L^2}$. Moreover, it is also a self-adjoint one, Taylor~\cite{taylor}. Indeed, the operators $L$ and $\widetilde{L}$ are self-adjoints (we can prove it, since they are maximal monotone symmetric operators, Brezis~\cite{brezis}, Serafini~\cite{serafini}).

The standard spectral theory for compact and self-adjoint operators seen in Theorem~\ref{diagon} applies in this context. We deduce that there exists a sequence of real number $\{\sigma_n \}_{n\in \N}$ such that $\{|\sigma_n|\}_{n\in \N}$ is non increasing and converging to zero. Moreover, if $\sigma_n$ and $\sigma_m$ are distinct, their corresponding eigenspaces are orthogonal. We call $\{w_{_{n}}\}_{n\in\N}$ the basis of eigenfunctions of $L^{-1}$.
So, we infer that $L$ has an orthonormal basis of $L^2(\Omega_l)\cup L^2(\Omega_r)$ of eigenfunctions $\{w_{_{n}}\}_{n\in\N}$ related to a sequence of increasing and diverging eigenvalues $\{\lambda_{_{n}}\}_{n\in\N}$ such that $\lambda_{_{n}} = \frac{1}{\sigma_n}$, for all $n\in\N$.
\begin{oss}
	The \hyperlink{mean}{mean zero property} can be interpreted as if we are taking the eigenfunctions in the orthogonal space of the constants. In fact, the existence of a sequence of eigenvalues and of orthogonal eigenfunctions in the diagonalization theorem can be proven through a minimisation process starting from the first zero eigenvalue and looking for the eigenspaces as the orthogonal spaces of its eigenfunction which is a constant.
\end{oss}

\section{Numerical method}\label{Tetameth}

We illustrate the one-dimension numerical method, Morton and Mayers~\cite{morton}, Quarteroni \textit{et al.}~\cite{qss}, used to perform the examples in Section~\ref{examples}.  
We present the discretization on the interval $I=(a,x_m)\cup (x_m,b)=:I_l \cup I_r$ of the one-dimension reaction-diffusion System~\eqref{eq}.

In the following, for simplicity, we write the numerical expressions for the equations of $u$, but with the same steps we can obtain the discretization also for $v$.
We consider a space discretization (see Almeida \textit{et al.}~\cite{almeida}) of each subdomain $I_l$ and $I_r$ in $N_l+1$ and $N_r+1$ points respectively. We observe
that this distinction allows to consider not centred membranes. In our case with the membrane in the middle point $x_m$, we infer that $N_l=N_r$. Concerning the membrane, the key aspect is to discretize this point as two distinct ones since the Kedem-Katchalsky conditions are constructed defining the right and left limit of the
density on the membrane (see Ciavolella and
Perthame~\cite{ciavper}). Moreover, the space step turns out to be $\Delta x = \frac{x_m-a}{N_l+1}= \frac{b-x_m}{N_r+1}$, with $N_l, N_r\in \N$. The mesh is formed by the intervals
$$I_i= \left(x_{i-\frac 1 2}, x_{i+\frac 1 2}\right), \; i=1,...,N_l+1, \qquad J_j= \left(x_{j-\frac 1 2}, x_{j+\frac 1 2}\right), \; j=1,...,N_r+1.$$
The intervals are centred in $x_i= i \Delta x$, $i=1,...,N_l+1$ and $x_j= j \Delta x$, $j=1,...,N_r+1$ with $I_{N_l+1}=J_1$. Moreover, as the reader can remark, we add ghost points to build the extremal intervals in the left $I_1, I_{N_l+1}$ and in the right $J_1, J_{N_r+1}$. Then, we consider the ghost points for $i=0, N_{l+2}$ and $j=0, N_{r+2}$.
At a given time, the spatial discretization of $u(t,x)$, interpreted in the finite volume sense, is of the form
\[u_i(t) \approx \frac{1}{\Delta x} \int_{I_i} u_l(t,x) \, dx, \qquad \h{u}_j(t) \approx \frac{1}{\Delta x} \int_{J_j} u_r(t,x) \, dx,\] 
for $i=1,...,N_l+1$ and $j=1,...,N_r+1$.
Concerning the time discretization, we consider the time step $\Delta t$ such that the mesh points are of the form $t^n = N_t \Delta t$, with $N_t\in \N$. The discrete approximation of $u(t,x)$, for $n\in \N$, $i=1,...,N_l+1$ and $j=1,...,N_r+1$, is now 
\[
u_i^n \approx \frac{1}{\Delta x} \int_{I_i} u_l(t^n,x) \, dx, \qquad \h{u}_j^n \approx \frac{1}{\Delta x} \int_{J_j} u_r(t^n,x) \, dx.
\]
We write the time discretization as an Euler method and the space one with a generic $\Theta$-method. In the simulations, we have chosen $\Theta=1$, meaning that the method is an implicit and always stable one. For the sake of simplicity, we consider a unique index $i$ instead of $i,j$. In the following, we take $0\leq n\leq N_t$ and we call $\delta_x^2 u^n_i= u^n_{i-1}-2u^n_i+u^n_{i+1}.$
Then, we obtain 
\begin{equation*}
		u^{n+1}_i-u^n_i=\mu_l[\Theta \, \delta_x^2 u^{n+1}_i + (1-\Theta) \delta_x^2 u^n_i] + \Delta t f^n_i, \quad \mbox{ for } i=1,...,N_l+1,
\end{equation*} 
with $\mu_l =\frac{D_{ul} \Delta t}{\Delta x ^2}$ and
\begin{equation*}
	\h{u}^{n+1}_i-\h{u}^n_i=\mu_r[\Theta \, \delta_x^2 \h{u}^{n+1}_i + (1-\Theta) \delta_x^2 \h{u}^n_i]+\Delta t \h{f}^n_i, \quad \mbox{ for } i=1,...,N_r+1,	
\end{equation*} 
with $\mu_r =\frac{D_{ur} \Delta t}{\Delta x ^2}$. Finally, we deduce the systems\\[1ex]
for $i=1,...,N_l+1$,
\begin{equation}\label{eqnum1}
	\begin{split}
		-\mu_l \Theta u_{i-1}^{n+1}&+ (1+2\mu_l\Theta) u_i^{n+1} -\mu_l \Theta u_{i+1}^{n+1}\\& = \mu_l(1-\Theta) u^n_{i-1}+(1-2\mu_l(1-\Theta)) u_i^{n}+\mu_l (1-\Theta) u_{i+1}^{n}+ \Delta t f^n_i
	\end{split}
\end{equation}
for $i=1,...,N_r+1$,
\begin{equation}\label{eqnum2}
	\begin{split}
		-\mu_r \Theta \h{u}_{i-1}^{n+1}&+ (1+2\mu_r\Theta) \h{u}_i^{n+1} -\mu_r \Theta \h{u}_{i+1}^{n+1} \\
		&= \mu_r(1-\Theta) \h{u}^n_{i-1}+(1-2\mu_r(1-\Theta)) \h{u}_i^{n}+\mu_r (1-\Theta) \h{u}_{i+1}^{n}+\Delta t \h{f}^n_i
	\end{split}
\end{equation}

Now, we exhibit the first order discretization of the boundary conditions. Starting from Neumann, we can distinguish the condition in $a$ and $b$ as
\begin{equation}\label{N}
	u_0^{n+1}=u_1^{n+1}, \qquad \h{u}_{N_r+2}^{n+1}=\h{u}_{N_r+1}^{n+1},
\end{equation} 
which give the relation of the extremal ghost points.
From the Kedem-Katchalsky membrane conditions, we deduce the expression of the membrane ghost points 
\begin{equation}\label{KK}
	u_{N_l+2}^{n+1}= u_{N_l+1}^{n+1}+\frac{\Delta x \, k_u}{D_{ul}}\, (\h{u}_1^{n+1}-u_{N_l+1}^{n+1}), \qquad \h{u}_{0}^{n+1}= u_{1}^{n+1}-\frac{\Delta x \, k_u}{D_{ur}}\, (\h{u}_1^{n+1}-u_{N_l+1}^{n+1}).
\end{equation}
Substituting the ghost values found in \eqref{N} and \eqref{KK} in the systems \eqref{eqnum1} and \eqref{eqnum2}, we get the equations at the extremal points:
\begin{description}
	\item[At the left limit on the membrane, ] 
	$$-\mu_l \Theta u_{N_l}^{n+1} +\left(1+\mu_l\Theta +\Theta \frac{\Delta t \, k_u}{\Delta x}\right) +u_{N_l+1}^{n+1}- \Theta \frac{\Delta t \, k_u}{\Delta x} \h{u}_{1}^{n+1}$$
	\begin{equation}\label{membrleft}
		= \mu_l (1-\Theta) u_{N_l}^{n} +\left(1-\mu_l(1-\Theta) -(1-\Theta) \frac{\Delta t \, k_u}{\Delta x}\right) +u_{N_l+1}^{n}+ (1-\Theta) \frac{\Delta t \, k_u}{\Delta x} \h{u}_{1}^{n}.
	\end{equation}
    \item[At the right limit on the membrane, ] 
    $$-\Theta \frac{\Delta t \, k_u}{\Delta x}\, u_{N_l+1}^{n+1} +\left(1+\mu_r\Theta +\Theta \frac{\Delta t \, k_u}{\Delta x}\right) \h{u}_{1}^{n+1}- \mu_r \Theta   \h{u}_{2}^{n+1}$$
    \begin{equation}\label{membrright}
    	= (1-\Theta) \frac{\Delta t \, k_u}{\Delta x}\, u_{N_l+1}^{n} +\left(1-\mu_r(1-\Theta) -(1-\Theta) \frac{\Delta t \, k_u}{\Delta x}\right) \h{u}_{1}^{n}+ \mu_r (1-\Theta)   \h{u}_{2}^{n}.
    \end{equation}
    \item[In $\mathbf{a}$,]
    \begin{equation}\label{a}
    	(1+\mu_l\Theta) u_1^{n+1} -\mu_l\Theta u_2^{n+1} = (1-\mu_l(1-\Theta)) u_1^n +\mu_l(1-\Theta) u_2^n.
    \end{equation}
    \item[In $\mathbf{b}$,]
    \begin{equation}\label{b}
    	-\mu_r\Theta \h{u}_{N_r}^{n+1}+ (1+\mu_r\Theta) \h{u}_{N_r+1}^{n+1} = \mu_r (1-\Theta) \h{u}_{N_r}^n + (1-\mu_r (1-\Theta)) \h{u}_{N_r+1}^n.
    \end{equation} 
\end{description}
To conclude, system \eqref{eqnum1} for $i=1,...,N_l$ and \eqref{eqnum2} for $i=1,...,N_r$, written for the internal points of the grid, combined with the equations for the extremal points \eqref{membrleft}, \eqref{membrright}, \eqref{a} and \eqref{b}, build the discretized system of $u$. The same equations with the proper coefficients can be found for $v$. 

Calling the vector solutions at time $t^n$ as
$$U^n = \begin{array}{cccccc}
	\left(u^n_1, \ldots, u^n_{N_l+1}, \h{u}^n_{1}, \ldots, \h{u}^n_{N_r+1}\right)^T,
\end{array} \qquad
V^n =\begin{array}{cccccc}
	\left(v^n_1, \ldots, v^n_{N_l+1}, \h{v}^n_{1}, \ldots, \h{v}^n_{N_r+1}\right)^T
\end{array}$$
and the reaction vectors as
$$F^{n} =\begin{array}{cccccc}
	\left(f^n_1,\, \ldots, f^n_{N_l+1}, \h{f}^n_{1}, \ldots, \h{f}^n_{N_r+1}\right)^T,
\end{array}
G^{n} =\begin{array}{cccccc}
	\left(g^n_1, \ldots, g^n_{N_l+1}, \h{g}^n_{1}, \ldots, \h{g}^n_{N_r+1}\right)^T,
\end{array}$$
we can write the discretized systems in a matrix form as $A U^{n+1}= B U^n + \Delta t F^n$ coupled with $C V^{n+1}= D V^n + \Delta t G^n$, where

\begin{center}
	A := \begin{tikzpicture}[baseline=(current bounding box.center)]
		\matrix (m) [matrix of math nodes,nodes in empty cells,right delimiter={)},left delimiter={(} ]{ 
			{\scriptstyle 1+\mu_l\Theta}  & {\scriptstyle -\mu_l \Theta}  &  {\scriptstyle 0}    &  &  &  &&&& {\scriptstyle 0} \\
			{\scriptstyle -\mu_l \Theta}    & {\scriptstyle 1+2\mu_l\Theta} &  & & & &&&&  \\
			{\scriptstyle 0} &&&& & & & & &    \\
			& & & {\scriptstyle 1+2\mu_l \Theta} & {\scriptstyle -\mu_l \Theta} & &&&&  \\
			& & & {\scriptstyle -\mu_l \Theta}  & {\scriptstyle 1+\mu_l \Theta+ \Theta \frac{\Delta t k_u}{\Delta x}} & {\scriptstyle -\Theta \frac{\Delta t k_u}{\Delta x} } &&&& \\
			& & &  & {\scriptstyle -\Theta \frac{\Delta t k_u}{\Delta x} } & {\scriptstyle 1+\mu_r \Theta+\Theta \frac{\Delta t k_u}{\Delta x}} & {\scriptstyle -\mu_r \Theta } &&&\\
			&&&&&{\scriptstyle -\mu_r \Theta } & {\scriptstyle 1+2\mu_r \Theta} & &&\\
			&&&&&&&&&{\scriptstyle 0}\\
			&&&&&&&& {\scriptstyle 1+2\mu_r \Theta} & {\scriptstyle -\mu_r \Theta}\\
			{\scriptstyle 0}&&&&&&&{\scriptstyle 0}& {\scriptstyle -\mu_r \Theta} & {\scriptstyle 1+\mu_r \Theta}\\
		} ;
		\draw[loosely dotted] (m-2-2)-- (m-4-4);
		\draw[loosely dotted] (m-1-2)-- (m-4-5);
		\draw[loosely dotted] (m-2-1)-- (m-5-4);
		\draw[loosely dotted] (m-7-6)-- (m-10-9);
		\draw[loosely dotted] (m-7-7)-- (m-9-9);
		\draw[loosely dotted] (m-6-7)-- (m-9-10);
		\draw[loosely dotted] (m-1-4)-- (m-1-9);
		\draw[loosely dotted] (m-2-10)-- (m-7-10);
		\draw[loosely dotted] (m-4-1)-- (m-9-1);
		\draw[loosely dotted] (m-10-2)-- (m-10-7);
	\end{tikzpicture}
\end{center}		
and, with the notation $\Theta':= 1-\Theta$, 
{\small
\begin{center}	
	B~:=~\begin{tikzpicture}[baseline=(current bounding box.center)]
		\matrix (m) [matrix of math nodes,nodes in empty cells,right delimiter={)},left delimiter={(} ]{ 
			{\scriptstyle 1-\mu_l\Theta'}  & {\scriptstyle \mu_l \Theta'}  &  {\scriptstyle 0}    &  &  &  &&&& {\scriptstyle 0} \\
			{\scriptstyle \mu_l \Theta'}    & {\scriptstyle 1-2\mu_l\Theta'} &  & & & &&&&  \\
			{\scriptstyle 0} &&&& & & & & &    \\
			& & & {\scriptstyle 1-2\mu_l \Theta'} & {\scriptstyle \mu_l \Theta'} & &&&&  \\
			& & & {\scriptstyle \mu_l \Theta'}  & {\scriptstyle 1-\mu_l \Theta'- \Theta' \frac{\Delta t k_u}{\Delta x}} & {\scriptstyle \Theta' \frac{\Delta t k_u}{\Delta x} } &&&& \\
			& & &  & {\scriptstyle \Theta' \frac{\Delta t k_u}{\Delta x} } & {\scriptstyle 1-\mu_r \Theta'-\Theta' \frac{\Delta t k_u}{\Delta x}} & {\scriptstyle \mu_r \Theta' } &&&\\
			&&&&&{\scriptstyle \mu_r \Theta' } & {\scriptstyle 1-2\mu_r \Theta'} & &&\\
			&&&&&&&&&{\scriptstyle 0}\\
			&&&&&&&& {\scriptstyle 1-2\mu_r \Theta'} & {\scriptstyle \mu_r \Theta'}\\
			{\scriptstyle 0}&&&&&&&{\scriptstyle 0}& {\scriptstyle \mu_r \Theta'} & {\scriptstyle 1-\mu_r \Theta'}\\
		} ;
		\draw[loosely dotted] (m-2-2)-- (m-4-4);
		\draw[loosely dotted] (m-1-2)-- (m-4-5);
		\draw[loosely dotted] (m-2-1)-- (m-5-4);
		\draw[loosely dotted] (m-7-6)-- (m-10-9);
		\draw[loosely dotted] (m-7-7)-- (m-9-9);
		\draw[loosely dotted] (m-6-7)-- (m-9-10);
		\draw[loosely dotted] (m-1-4)-- (m-1-9);
		\draw[loosely dotted] (m-2-10)-- (m-7-10);
		\draw[loosely dotted] (m-4-1)-- (m-9-1);
		\draw[loosely dotted] (m-10-2)-- (m-10-7);
	\end{tikzpicture}.
\end{center}}
\noindent Substituting $\mu_l, \mu_r, k_u$ with the notation $\sigma_l =\frac{D_{vl} \Delta t}{\Delta x ^2}$, $ \sigma_r =\frac{D_{vr} \Delta t}{\Delta x ^2}$ and $k_v$, we can write the matrix $C$ and $D$.
We report the core of the Matlab code here below.

{\footnotesize\lstinputlisting[caption = {Matlab code}]{numer2905.m}}
\bibliographystyle{abbrvnat}
\bibliography{references_chapter}
\end{document}